\documentclass[11pt]{article}
\usepackage[margin=1in]{geometry} 
\usepackage{amssymb,amsfonts,amsmath,bbm,mathrsfs,stmaryrd}
\usepackage{xcolor}
\usepackage{url}

\usepackage{enumerate}

\usepackage[colorlinks,
             linkcolor=black!75!red,
             citecolor=blue,
             pdftitle={},
             pdfauthor={},
             pdfproducer={pdfLaTeX},
             pdfpagemode=None,
             bookmarksopen=true
             bookmarksnumbered=true]{hyperref}

\usepackage{tikz}
\usetikzlibrary{arrows,calc,decorations.pathreplacing,decorations.markings,intersections,shapes.geometric,through,fit,shapes.symbols,positioning,decorations.pathmorphing}

\usepackage{braket}

\usepackage[amsmath,thmmarks,hyperref]{ntheorem}
\usepackage{cleveref}

\creflabelformat{enumi}{#2(#1)#3}

\crefname{section}{Section}{Sections}
\crefformat{section}{#2Section~#1#3} 
\Crefformat{section}{#2Section~#1#3} 

\crefname{subsection}{\S}{\S\S}
\AtBeginDocument{%
  \crefformat{subsection}{#2\S#1#3}%
  \Crefformat{subsection}{#2\S#1#3}% 
}

\theoremstyle{plain}

\newtheorem{lemma}{Lemma}[section]
\newtheorem{proposition}[lemma]{Proposition}
\newtheorem{corollary}[lemma]{Corollary}
\newtheorem{theorem}[lemma]{Theorem}

\theoremstyle{nonumberplain}

\theoremstyle{plain}
\theorembodyfont{\upshape}
\theoremsymbol{\ensuremath{\blacklozenge}}

\newtheorem{definition}[lemma]{Definition}

\newtheorem{remark}[lemma]{Remark}

\crefname{definition}{definition}{definitions}
\crefformat{definition}{#2definition~#1#3} 
\Crefformat{definition}{#2Definition~#1#3} 

\crefname{ex}{example}{examples}
\crefformat{example}{#2example~#1#3} 
\Crefformat{example}{#2Example~#1#3} 

\crefname{remark}{remark}{remarks}
\crefformat{remark}{#2remark~#1#3} 
\Crefformat{remark}{#2Remark~#1#3} 

\crefname{convention}{convention}{conventions}
\crefformat{convention}{#2convention~#1#3} 
\Crefformat{convention}{#2Convention~#1#3}

\crefname{lemma}{lemma}{lemmas}
\crefformat{lemma}{#2lemma~#1#3} 
\Crefformat{lemma}{#2Lemma~#1#3} 

\crefname{proposition}{proposition}{propositions}
\crefformat{proposition}{#2proposition~#1#3} 
\Crefformat{proposition}{#2Proposition~#1#3} 

\crefname{corollary}{corollary}{corollaries}
\crefformat{corollary}{#2corollary~#1#3} 
\Crefformat{corollary}{#2Corollary~#1#3} 

\crefname{theorem}{theorem}{theorems}
\crefformat{theorem}{#2theorem~#1#3} 
\Crefformat{theorem}{#2Theorem~#1#3} 

\crefname{enumi}{}{}
\crefformat{enumi}{(#2#1#3)}
\Crefformat{enumi}{(#2#1#3)}

\crefname{assumption}{assumption}{Assumptions}
\crefformat{assumption}{#2assumption~#1#3} 
\Crefformat{assumption}{#2Assumption~#1#3} 

\crefname{equation}{}{}
\crefformat{equation}{(#2#1#3)} 
\Crefformat{equation}{(#2#1#3)}

\numberwithin{equation}{section}
\renewcommand{\theequation}{\thesection-\arabic{equation}}

\theoremstyle{nonumberplain}
\theoremsymbol{\ensuremath{\blacksquare}}

\newtheorem{proof}{Proof}
\newcommand\pf[1]{\newtheorem{#1}{Proof of \Cref{#1}}}

\newcommand\bC{{\mathbb C}}

\newcommand\bR{{\mathbb R}}
\newcommand\bS{{\mathbb S}}

\newcommand\cA{{\mathcal A}}
\newcommand\cB{{\mathcal B}}
\newcommand\cC{{\mathcal C}}

\newcommand\cH{{\mathcal H}}
\newcommand\cI{{\mathcal I}}

\newcommand\cK{{\mathcal K}}
\newcommand\cL{{\mathcal L}}

\newcommand\cQ{{\mathcal Q}}

\newcommand{\la}{\leftarrow}

\newcommand{\ot}{\otimes}

\newcommand{\raro}{\rightarrow}

\newcommand{\clb}{{\cal B}}
\newcommand{\clc}{{\cal C}}

\newcommand{\clq}{{\cal Q}}

\DeclareMathOperator{\id}{id}

\newcommand\numberthis{\addtocounter{equation}{1}\tag{\theequation}}

\newcommand{\cat}[1]{\textsc{#1}}

\newcommand{\qedhere}{\mbox{}\hfill\ensuremath{\blacksquare}}

\title{Existence and rigidity of quantum isometry groups for compact metric spaces}
\author{Alexandru Chirvasitu\footnote{Partially supported by NSF grant DMS-1801011} and Debashish Goswami\footnote{Partially supported by J.C. Bose Fellowship from D.S.T. (Govt. of India).}}

\begin{document}

\date{}

\newcommand{\Addresses}{{% additional braces for segregating \footnotesize
  \bigskip
  \footnotesize

  \textsc{Department of Mathematics, University at Buffalo, Buffalo,
    NY 14260-2900, USA}\par\nopagebreak \textit{E-mail address}:
  \texttt{achirvas@buffalo.edu}

  \medskip
  
  \textsc{Indian Statistical Institute,
    203, B. T. Road, Kolkata 700108, India}\par\nopagebreak \textit{E-mail address}:
  \texttt{goswamid@isical.ac.in}

}}

\maketitle

\begin{abstract}
  We prove the existence of a quantum isometry groups for new classes of metric spaces: (i) geodesic metrics for compact connected Riemannian manifolds (possibly with boundary) and (ii) metric spaces admitting a uniformly distributed probability measure. In the former case it also follows from recent results of the second author that the quantum isometry group is classical, i.e. the commutative $C^*$-algebra of continuous functions on the Riemannian isometry group.
\end{abstract}


\vspace{.5cm}


%\tableofcontents

%%%%%%%%%%%%%%%%%%%%%%%%%%%%%%%%%%%%%%%%%%%%%%%%%%%%%%%%%%%%%%%%%%%%%%%%%%%%%%%%%%%%%%%%%%%%%%%%%%%%%%%%%%%%%%%%%%
%%%%%%%%%%%%%%%%%%%%%%%%%%%%%%%%%%%%%%%%%%%%%%%%%%%%%%%%%%%%%%%%%%%%%%%%%%%%%%%%%%%%%%%%%%%%%%%%%%%%%%%%%%%%%%%%%%
\section*{Introduction}

Having originated in the mathematical physics literature \cite{drinfeld,jimbo,frt,soi}, quantum groups now constitute a rich and actively-developed field. While the original impetus was mainly algebraic in nature, further developments have given the topic a functional-analytic flavor through the work of Woronowicz \cite{Pseudogroup}, Podles \cite{Podles}, Kustermans-Vaes \cite{kustermans} and many more (too numerous to do justice here).

Actions of quantum groups are typically cast as coactions of certain Hopf algebras on algebraic or geometric structures, in the style of Manin's study \cite{manin_book} of quantum symmetries for quadratic graded algebras. In the framework introduced in \cite{Pseudogroup} the types of structures whose quantum symmetries one is led to consider abound: finite (quantum) graphs, finite non-commutative measure spaces (i.e. finite-dimensional $C^*$-algebras equipped with distinguished states, finite metric spaces, etc.). We refer the reader to \cite{ban_1,bichon,wang,ban_col} for some (of the numerous) examples.

In the same spirit, the second author introduced in \cite{Goswami} the concept of quantum automorphism group of a spectral triple, the latter being an incarnation of a Riemannian or spin manifold in Connes' framework for non-commutative geometry \cite{con}. The topic has provided a rich supply of problems and examples, as reflected by further work on it \cite{skalski_bhow,Soltan}.

In the present paper we are concerned with quantum symmetries of {\it classical} structures, specifically compact metric spaces. One phenomenon that has emerged from recent work in the field is that certain ``sufficiently regular'' classical structures are quantum-rigid, in the sense that a compact quantum group acting faithfully in a structure-preserving manner is automatically classical, i.e. a plain compact group. The recent \cite{final} confirms a conjecture to that effect by the second author:

\begin{theorem}[3.10 of \cite{final}]
  A compact quantum group acting faithfully and smoothly on a closed connected smooth manifold is classical.
\end{theorem}

In this paper we prove a stronger version of the above theorem by weakening the smoothness condition to what we have termed `weak smoothness'. This keeps with the spirit of similar rigidity results in slightly varying settings:

\begin{enumerate}[(1)]
\item An analogue under the additional assumption that the action preserves the Laplacian of a Riemannian metric \cite{gafa}. 
\item A semisimple and cosemisimple Hopf algebra (hence also finite-dimensional) coacting faithfully on a commutative domain must be commutative \cite{Etingof_walton_1}.
\item An isometric faithful action of a compact quantum group on the geodesic metric space of a negatively-curved connected closed Riemannian manifold is classical \cite{chirvasitu}.
\end{enumerate}

This last result is placed in the context of isometric actions as introduced in \cite{metric} and will be generalized in some of our main results below (\Cref{th.nbdry,th.rig-bdry}):

\begin{theorem}
  A compact quantum group acting isometrically on the geodesic metric space of a compact connected Riemannian manifold is classical.
\end{theorem}

On a somewhat different note, a phenomenon that has received some attention in the literature is the problem of whether or not a given piece of structure even {\it has} a quantum automorphism group: a ``largest'' or universal quantum group acting in a structure-preserving manner.

The issue was first illustrated in \cite[Theorem 6.1]{wang}: although a finite classical space $X$ admits a quantum automorphism group that automatically preserves the uniform measure on $X$, in general a finite-dimensional $C^*$-algebra $A$ does not admit such a universal action. The problem is that every compact quantum group acting on $A$ will automatically preserve a state on $A$, but there is no ``canonical'' state preserved by all such actions.

For essentially the same reason, it is unclear whether, for a given compact metric space $(X,d)$, there is a universal compact quantum group acting isometrically on $X$ in the sense of \cite[Definition 3.1]{metric}. Contrast this with classical group actions: the isometry group of a compact metric space is compact, and hence is universal among classical compact groups acting isometrically.

As in the case of finite-dimensional algebras touched on above, it is not difficult to show that having fixed a probability measure $\mu$ on $X$, there {\it is} a universal compact quantum group $QAUT(X,d,\mu)$ among those that act on $X$ so as to preserve both $d$ and $\mu$. As before, it is unclear in general how to select a ``best'' measure $\mu$ preserved by every quantum action in order to construct a universal quantum isometry group $QAUT(X,d)$. The choice, however, is obvious when the metric space $(X,d)$ admits a {\it uniformly distributed measure} (see \Cref{def.ud}): one which assigns equal mass to balls of equal radii.

It is well known that uniformly distributed probability measures are unique when they exist. In that case we have (see \Cref{th.ud}):

\begin{theorem}
  Let $(X,d)$ be a compact metric space admitting a uniformly distributed probability measure $\mu$. Then, every compact quantum group acting isometrically on $(X,d)$ leaves $\mu$ invariant.
\end{theorem}

Coupling this with the previous remarks on the existence of $QAUT(X,d,\mu)$, it follows that all such metric spaces $(X,d)$ have quantum isometry groups. These need not be classical, in general: perhaps the ``simplest'' example is the quantum symmetric group $S_n^+$ introduced in \cite[\S 3]{wang}: it can be recast as $QAUT(X,d)$ where
\begin{equation*}
  X=\{1,\cdots,n\}
\end{equation*}
and $d$ is the uniform distance:
\begin{equation*}
  d(i,j)=
  \begin{cases}
    0&\text{if } i=j\\
    1&\text{otherwise}
  \end{cases}
\end{equation*}

The paper is organized as follows.

\Cref{se.prel} recalls some background needed later, on the various topics we touch on (compact quantum groups, their actions, Riemannian geometry, etc.).

In \Cref{se.smth} we prove some preliminary results on smooth actions, building on some of the material from \cite{gafa,final}.

Finally, \Cref{se.main} contains the main results of the paper. \Cref{th.nbdry} proves that faithful isometric quantum actions on connected closed Riemannian manifolds are classical and \Cref{th.rig-bdry} extends this to compact connected manifolds with boundary. In the course of unwinding the argument we prove other results that might be of some independent interest:
\begin{itemize}
\item Recall that a homeomorphism of a topological manifold automatically preserves its boundary. We prove in \Cref{pr.bdry-inv} that similarly, a quantum isometric action on a compact connected manifold leaves the boundary invariant. 
\item We also prove in \Cref{cor.faith} that (once more, as expected from the classical situation) if a quantum isometric action as above is faithful and all connected components of the compact manifold acted upon have non-empty boundary then the restriction of the action to the boundary is again faithful.

\item In \Cref{th.doubled} we extend a quantum action $\alpha$ on a compact manifold with boundary to the {\it double} $M\cup_{\partial M}M$ of the manifold in the sense of \cite[Example 9.32]{lee} and show that the doubled action retains some of the relevant properties of $\alpha$.
\end{itemize}

Finally, in \Cref{subse.unif} we prove that compact metric spaces which admit uniformly distributed probability measures have quantum isometry groups.

%%%%%%%%%%%%%%%%%%%%%%%%%%%%%%%%%%%%%%%%%%%%%%%%%%%%%%%%%%%%%%%%%%%%%%%%%%%%%%%%%%%%%%%%%%%%%%%%%%%%%%%%%%%%%%%%%%
\subsection*{Acknowledgements}

We would like to express our warm thanks to the anonymous referees for a very thorough treatment of the initial draft; we believe their suggestions have improved the material considerably.

%%%%%%%%%%%%%%%%%%%%%%%%%%%%%%%%%%%%%%%%%%%%%%%%%%%%%%%%%%%%%%%%%%%%%%%%%%%%%%%%%%%%%%%%%%%%%%%%%%%%%%%%%%%%%%%%%%
%%%%%%%%%%%%%%%%%%%%%%%%%%%%%%%%%%%%%%%%%%%%%%%%%%%%%%%%%%%%%%%%%%%%%%%%%%%%%%%%%%%%%%%%%%%%%%%%%%%%%%%%%%%%%%%%%%
\section{Preliminaries}\label{se.prel}

%%%%%%%%%%%%%%%%%%%%%%%%%%%%%%%%%%%%%%%%%%%%%%%%%%%%%%%%%%%%%%%%%%%%%%%%%%%%%%%%%%%%%%%%%%%%%%%%%%%%%%%%%%%%%%%%%%
\subsection{Notational conventions}

We write $\cB(\cH)$ for the algebra of bounded operators on a Hilbert space $\cH$ and $\cB_0(\cH)$ for the ideal of compact operators. $Sp$, $\overline{Sp}$ denote the linear span and respectively the closed linear span of elements of a vector space (closed in whatever topology is relevant to the discussion).

Several flavors of tensor products appear below:
\begin{itemize}
\item $\otimes$ is the minimal tensor product between $C^*$-algebras and more generally locally convex spaces and on one occasion, the spatial tensor product between von Neumann algebras.
\item $\overline{\otimes}$ stands for the tensor product of Hilbert spaces and modules. 
\item $\otimes_{\rm alg}$ is the algebraic tensor product between vector spaces, non-topological algebras, etc.
\item $T\otimes S$ denotes the tensor product of maps $S$ and $T$ in {\it all} of the above-mentioned cases.
\end{itemize}

We denote by $C(X)$ or $C^{\infty}(X)$ the spaces of continuous and smooth complex-valued functions on $X$ respectively and add an `$\bR$' to indicate real-valued functions, as in $C^{\infty}(X,\bR)$. 

%A scalar valued inner product of Hilbert spaces will be denoted by $<\cdot, \cdot >$ and some (non-scalar) $\ast$-algebra valued inner product of Hilbert modules over locally convex $\ast$-algebras will be denoted by $\lgl\lgl \cdot, \cdot \rgl\rgl$.  For a Hilbert $\cla$-module $E$ where $\cla$ is a $C^*$ algebra, we denote the $C^*$-algebra of adjointable right $\cla$-linear maps by $\cll(E)$. In particular, we'll consider the trivial Hilbert modules of the form $\clh \overline{\ot} \cla$.

%%%%%%%%%%%%%%%%%%%%%%%%%%%%%%%%%%%%%%%%%%%%%%%%%%%%%%%%%%%%%%%%%%%%%%%%%%%%%
\subsection{Compact quantum groups and their actions}\label{subse:cqg}

We need some basic material on compact quantum groups and their actions on non-commutative spaces, as covered, say, in \cite{Van,Pseudogroup,Wor98}. The present section serves to recall some of this material.

A compact quantum group (CQG for short) is a unital $C^{\ast}$ algebra $\clq$ equipped with a $C^*$-algebra morphism $\Delta$, {\it coassociative} in the sense that
\begin{equation*}
  \begin{tikzpicture}[auto,baseline=(current  bounding  box.center)]
    \path[anchor=base] (0,0) node (1) {$\cQ$} +(2,.5) node (2) {$\cQ\otimes \cQ$} +(2,-.5) node (3) {$\cQ\otimes \cQ$} +(4,0) node (4) {$\cQ\otimes\cQ\otimes \cQ$};
  \draw[->] (1) to[bend left=6] node[pos=.5,auto] {$\scriptstyle \Delta$} (2);
  \draw[->] (1) to[bend right=6] node[pos=.5,auto,swap] {$\scriptstyle \Delta$} (3);
  \draw[->] (2) to[bend left=6] node[pos=.5,auto] {$\scriptstyle \Delta\otimes\id$} (4);
  \draw[->] (3) to[bend right=6] node[pos=.5,auto,swap] {$\scriptstyle \id\otimes \Delta$} (4);      
\end{tikzpicture}
\end{equation*}
commutes and such that 
\begin{equation*}
  \Delta(\clq)(\clq\ot 1),\quad \Delta(\clq)(1\ot \clq)\quad \subseteq\quad \cQ\otimes \cQ
\end{equation*}
are both norm-dense. This suffices to ensure the existence of a unique dense Hopf $*$-subalgebra $\cQ_0\subseteq \cQ$, equipped with a counit $\varepsilon:\cQ_0\to \bC$ and an antipode $\kappa:\cQ_0\to \cQ_0$.

For every compact quantum group $\cQ$ the convolution multiplication
\begin{equation*}
 \begin{tikzpicture}[auto,baseline=(current  bounding  box.center)]
   \path[anchor=base] (0,0) node (1) {$\cQ$} +(2,.5) node (2) {$\cQ\otimes \cQ$} +(4,0) node (3) {$\bC$}; \draw[->] (1) to[bend left=6] node[pos=.5,auto] {$\scriptstyle \Delta$} (2); \draw[->] (2) to[bend left=6] node[pos=.5,auto] {$\scriptstyle \varphi\otimes\psi$} (3); \draw[->] (1) to[bend right=6] node[pos=.5,auto,swap] {$\scriptstyle \varphi*\psi$} (3);
 \end{tikzpicture}
\end{equation*}
of states $\varphi$ and $\psi$ makes the state space $S(\cQ)$ (or $\mathrm{Prob}(\cQ)$) of $\cQ$ into a semigroup (or monoid if $\cQ$ has a bounded counit).

A compact quantum group $\cQ$ has a unique {\it Haar state} $h$ characterized by the fact that it ``absorbs'' every other state under convolution:
\begin{equation*}
  \varphi*h = h*\varphi = h,\ \forall \varphi\in S(\cQ). 
\end{equation*}

A compact quantum group is {\it reduced} if its Haar state is faithful. Every compact quantum group $\cQ$ has a reduced version $\cQ_r$ defined as the image of the GNS representation of the Haar state. The comultiplication of $\cQ$ descends through the quotient $\cQ_r$, making the latter into a CQG again.

\begin{definition}\label{def:act}
\label{def.CQG_action}
A unital $\ast$-homomorphism $\alpha:\clc\raro \clc \ot \clq$, where $\clc$ is a unital $C^\ast$-algebra and $\clq$ is a CQG, is said to be an action of $\clq$ on $\clc$ if
 \begin{enumerate}[(1)]
 \item\label{item:11} the diagram
   \begin{equation*}
 \begin{tikzpicture}[auto,baseline=(current  bounding  box.center)]
  \path[anchor=base] (0,0) node (1) {$\cC$} +(2,.5) node (2) {$\cC\otimes \cQ$} +(2,-.5) node (3) {$\cC\otimes \cQ$} +(4,0) node (4) {$\cC\otimes\cQ\otimes \cQ$};
  \draw[->] (1) to[bend left=6] node[pos=.5,auto] {$\scriptstyle \alpha$} (2);
  \draw[->] (1) to[bend right=6] node[pos=.5,auto,swap] {$\scriptstyle \alpha$} (3);
  \draw[->] (2) to[bend left=6] node[pos=.5,auto] {$\scriptstyle \alpha\otimes\id$} (4);
  \draw[->] (3) to[bend right=6] node[pos=.5,auto,swap] {$\scriptstyle \id\otimes \Delta$} (4);      
 \end{tikzpicture}
\end{equation*}
commutes (co-associativity) and
 \item\label{item:12} $Sp \ \alpha(\clc)(1\ot \clq)$ is norm-dense in $\clc \ot \clq$.    
 \end{enumerate}
 
 $\alpha$ is {\it faithful} if
 \begin{equation*}
   Sp\{(\varphi\otimes\id)\alpha(x)\ |\ x\in \cC,\ \varphi\in S(\cC)\}
 \end{equation*}
generates $\cQ$ as a $C^*$-algebra. 
\end{definition}
An action $\alpha$ as in \Cref{def.CQG_action} induces a right action of the semigroup $S(\cQ)$ introduced above on the state space $S(\cC)$ of $\cC$, denoted by $\triangleleft$ and defined by
\begin{equation}\label{eq:12}
 \begin{tikzpicture}[auto,baseline=(current  bounding  box.center)]
  \path[anchor=base] (0,0) node (1) {$\cC$} +(2,.5) node (2) {$\cC\otimes \cQ$} +(4,0) node (3) {$\bC$.};
  \draw[->] (1) to[bend left=6] node[pos=.5,auto] {$\scriptstyle \alpha$} (2);
  \draw[->] (2) to[bend left=6] node[pos=.5,auto] {$\scriptstyle \varphi\otimes\psi$} (3);
  \draw[->] (1) to[bend right=6] node[pos=.5,auto,swap] {$\scriptstyle \varphi\triangleleft\psi$} (3);    
 \end{tikzpicture}
\end{equation}

An action $\alpha$ of $\clq$ on $\clc$ induces an action $\alpha_r$ by the reduced version $\clq_r$ of $\cQ$:
\begin{equation*}
 \begin{tikzpicture}[auto,baseline=(current  bounding  box.center)]
  \path[anchor=base] (0,0) node (1) {$\cC$} +(2,.5) node (2) {$\cC\otimes \cQ$} +(4,0) node(3){$\cC\otimes \cQ_r$,};
  \draw[->] (1) to[bend left=6] node[pos=.5,auto] {$\scriptstyle \alpha$} (2);
  \draw[->] (2) to[bend left=6] node[pos=.5,auto] {$\scriptstyle \id\otimes \pi_{\cQ}$} (3);
  \draw[->] (1) to[bend right=6] node[pos=.5,auto,swap] {$\scriptstyle \alpha_r$} (3);    
 \end{tikzpicture}
\end{equation*}
where $\pi_{\cQ}:\cQ\to \cQ_r$ is the canonical surjection. The original action $\alpha$ is faithful if and only if $\alpha_r$ is.

For every action $\alpha$ there is a dense $*$-subalgebra $\cC_0\subseteq \cC$ on which $\alpha$ restricts to a purely algebraic coaction of the Hopf algebra $\cQ_0\subseteq \cQ$:
\begin{equation*}
 \begin{tikzpicture}[auto,baseline=(current  bounding  box.center)]
  \path[anchor=base] (0,0) node (1) {$\cC_0$} +(2,.5) node (mu) {$\cC$} +(2,-.5) node (md) {$\cC_0\otimes_{\rm alg}\cQ_0$} +(4,0) node (2) {$\cC\otimes \cQ$,};
  \draw[right hook->] (1) to[bend left=6] node[pos=.5,auto] {$\scriptstyle $} (mu);
  \draw[->] (mu) to[bend left=6] node[pos=.5,auto] {$\scriptstyle \alpha$} (2);  
  \draw[right hook->] (md) to[bend right=6] node[pos=.5,auto] {$\scriptstyle $} (2);
  \draw[->] (1) to[bend right=6] node[pos=.5,auto] {$\scriptstyle $} (md);    
 \end{tikzpicture}
\end{equation*}
where the hooked arrows are the obvious inclusions.

% % We use Sweedler notation for both comultiplications and coactions:
% % \begin{equation*}
% %   \Delta(a) = a_1\otimes a_2,\quad \alpha(x) = x_0\otimes x_1
% % \end{equation*}
% % (so that summation is suppressed).
% %

Following \cite{Wor98} (or rather paraphrasing it), recall that a {\it unitary representation} of a CQG $(\cQ,\Delta)$ on a Hilbert space $\cH$ is a unitary $U$ in the space $\cL(\cH\overline{\otimes}\cQ)$ of adjointable operators on the Hilbert $\cQ$-module $\cH\overline{\otimes}\cQ$ such that the linear map
\begin{equation*}
  V:\cH\to \cH\overline{\otimes}\cQ,\quad V(\xi):=U(\xi\otimes 1)
\end{equation*}
makes
\begin{equation*}
  \begin{tikzpicture}[auto,baseline=(current  bounding  box.center)]
    \path[anchor=base] (0,0) node (1) {$\cH$} +(2,.5) node (2) {$\cH\overline{\otimes} \cQ$} +(2,-.5) node (3) {$\cH\overline{\otimes} \cQ$} +(4,0) node (4) {$\cH\overline{\otimes}\cQ\otimes \cQ$};
    \draw[->] (1) to[bend left=6] node[pos=.5,auto] {$\scriptstyle V$} (2);
    \draw[->] (1) to[bend right=6] node[pos=.5,auto,swap] {$\scriptstyle V$} (3);
    \draw[->] (2) to[bend left=6] node[pos=.5,auto] {$\scriptstyle V\otimes\id$} (4);
    \draw[->] (3) to[bend right=6] node[pos=.5,auto,swap] {$\scriptstyle \id\otimes \Delta$} (4);      
  \end{tikzpicture}
\end{equation*}
commute. 

\begin{definition}\label{def.impl}
  An action $\alpha$ as in \Cref{def.CQG_action} is {\it implemented} by a unitary representation of $\cQ$ on $\cH$ if we can represent
  \begin{equation*}
    \pi:\cC\subset \cB(\cH)
  \end{equation*}
faithfully on a Hilbert space such that   
\begin{equation*}
  \alpha(a)=U(\pi(a)\otimes 1)U^{-1}
\end{equation*}
for all $a\in \cC$
\end{definition}

It is not difficult to see that if $\alpha$ is implemented by a unitary representation $U$ then it is one-to-one (or injective). We can say even more: $U$ induces a unitary representation
\begin{equation*}
  U_r:=(\id\otimes \pi_{\cQ})U
\end{equation*}
of $\cQ_r$, where $\pi_{\cQ}:\cQ\to \cQ_r$ is reduction surjection. It is then easy to check that $U_r$ implements the reduced counterpart $\alpha_r$ of $\alpha$, and hence $\alpha_r$ is injective.  

The converse also holds: an injective reduced action $\alpha$ is implemented by a unitary representation $U$ of $\cQ$. To see this, consider the family of all states $\varphi_i$, $i\in I$ on $\cC$. Since
\begin{equation*}
  \alpha:\cC\to \cC\otimes \cQ
\end{equation*}
is an embedding, the compositions
\begin{equation*}
  \overline{\varphi_i}:=\varphi_i\triangleleft h = (\varphi_i\otimes h)\circ \alpha
\end{equation*}
defined in \Cref{eq:12} form a jointly faithful family on $\cC$, and hence the direct sum of their attached GNS representations
\begin{equation*}
  \rho_i:\cC\to \cB(L^2(\cC,\overline{\varphi_i}))
\end{equation*}
is faithful. Furthermore, because each $\overline{\varphi_i}$ is {\it invariant} under $\alpha$ in the sense that
\begin{equation*}
  \overline{\varphi_i}\triangleleft \psi = \overline{\varphi_i}
\end{equation*}
for all states $\psi\in S(\cQ)$, the map
\begin{equation*}
  a\otimes q\mapsto \alpha(a)(1\otimes q),\ a\in \cC,\ q\in \cQ
\end{equation*}
extends to a unitary representation of $\cQ$ on the underlying space
\begin{equation*}
  \bigoplus_{i\in I} L^2(\cC,\overline{\varphi_i})
\end{equation*}
of the direct sum representation $\bigoplus_i \rho_i$ which as desired, implements $\alpha$.

When $\cC$ is classical, i.e. $C(X)$ for a compact Hausdorff space $X$, the invariant states $\overline{\varphi_i}$ are probability measures $\mu_i$ on $X$ and hence $\alpha$ is induced by a unitary representation of $\cQ$ on the direct sum of Hilbert spaces $L^2(X,\mu_i)$. Furthermore, when $X$ admits a {\it single} faithful probability measure (e.g. if $X$ is metrizable) then we can represent $\cQ$ on a single Hilbert space $L^2(X,\mu)$. 

%%%%%%%%%%%%%%%%%%%%%%%%%%%%%%%%%%%%%%%%%%%%%%%%%%%%%%%%%%%%%%%%%%%%%%%%%%%%%
\subsection{Isometric actions}\label{subse.metric}

Let $(X,d)$ be a compact metric space and $\cQ$ a compact quantum group acting faithfully on $X$. 

We always assume $\cQ$ has a bounded antipode whenever referring to isometric actions. This is mostly harmless in our circumstances: according to \cite[Theorem 3.16]{hua-inv} compact quantum groups acting faithfully on a unital $C^*$-algebra so as to preserve a tracial state are automatically {\it of Kac type} in the sense that their antipodes are involutive ($\kappa^2=\id$) on the unique dense Hopf subalgebra of $\cQ$. $\kappa$ then descends to a bounded multiplication-reversing $*$-automorphism of the reduced counterpart $\cQ_r$ of $\cQ$ and we can always pass to the reduced version $\alpha_r$ of the action $\alpha$.

We follow \cite[Definition 3.1 and Lemma 3.2]{metric} in defining the notion of an isometric action of a compact quantum group $\cQ$ on $X$:

\begin{definition}\label{def.isometric}
  Let $(X,d)$ be a compact metric space and write $d_x$ for the function $d(x,-)$ and $\cC=C(X)$.
  
  A faithful action $\alpha:\cC\to \cC\otimes \cQ$ is {\it isometric} if
  \begin{equation*}
    \alpha(d_x)(y)=\kappa(\alpha(d_y)(x))
  \end{equation*}
  for all $x,y\in X$, where $\kappa$ is the bounded antipode of $\cQ$.
\end{definition}

Note that if $\alpha$ is isometric then so is $\alpha_r$, and moreover by \cite[Proposition 3.10]{Chi15}
\begin{equation*}
  \alpha_r:\cC\to \cC\otimes \cQ
\end{equation*}
is one-to-one. We will make crucial use of this below.

%%%%%%%%%%%%%%%%%%%%%%%%%%%%%%%%%%%%%%%%%%%%%%%%%%%%%%%%%%%%%%%%%%%%%%%%%%%%% 
\subsection{Manifolds with boundary}\label{subse:prel-bdry}

In \Cref{subse.bdry} we work extensively with (compact) smooth and Riemannian manifolds with boundary. Much of the general background extends from the boundaryless case without issue, but it is sometimes difficult to locate appropriate references in the literature. With that in mind, we give a few references here.

\cite{lee} is a good overall source, given that care is taken throughout to phrase results so that they apply to manifolds with boundary. In particular, smooth structures on such manifolds are introduced on \cite[pp.27-29]{lee} and Riemannian structures in \cite[Chapter 13]{lee}. Assuming for simplicity that our manifold is embedded into a Euclidean space of the same dimension (the reasoning carries through in general by picking coordinate patches, etc.), the discussion on \cite[p.27]{lee} shows that the following two conditions on a function $f:M\to \bR$ (or $\bC$) are equivalent:
\begin{itemize}
\item $f$ extends to a smooth map on an open neighborhood of $M$;
\item $f$ is continuous on $M$, smooth in the interior of $M$, and all of its partial derivatives extend continuously to the boundary $\partial M$. 
\end{itemize}

These are the functions one naturally regards as smooth on $M$ even when $\partial M\ne \emptyset$, and we denote the algebra they constitute by
\begin{itemize}
\item $C^{\infty}(M)$ for complex-valued functions;
\item $C^{\infty}(M,\bR)$ in the real-valued case. 
\end{itemize}
We focus on the former to fix ideas, but everything of substance mentioned below is valid for real-valued functions. 

The suprema of all of the partial derivatives form a family of seminorms making $C^{\infty}(M)$ into a locally convex topological vector space, and one proves as ``usual'' (i.e. in the boundaryless case) that $C^{\infty}(M)$ is nuclear (e.g. \cite[Corollary to Theorem 51.5]{trev}). 

We will also work with smooth maps on $M$ valued in a $C^*$-algebra (like, say, the compact quantum group function algebras $\cQ$ discussed above). We denote these by
\begin{equation*}
  C^{\infty}(M,\cQ):=C^{\infty}(M)\otimes \cQ,
\end{equation*}
where the tensor of locally convex vector spaces is unambiguous by nuclearity. Concretely, the elements of $C^{\infty}(M,\cQ)$ are those functions $M\to \cQ$ which
\begin{itemize}
\item are continuous on $M$;
\item smooth in the interior of $M$ in the usual sense that all higher derivatives exist;
\item admit continuous extensions of all partial derivatives to the boundary $\partial M$. 
\end{itemize}

We will need the following version of Nachbin's approximation theorem for algebras of smooth functions (e.g. \cite[Theorem 1.2.1]{llav}), which is usually stated as the equivalence of \Cref{item:7} and \Cref{item:8} below.

\begin{theorem}\label{th:nachb}
  Let $M$ be an $n$-dimensional compact smooth manifold and $\cA\subseteq C^{\infty}(M,\bR)$ a unital subalgebra. The following conditions are equivalent:
  \begin{enumerate}[(1)]
  \item\label{item:7} $\cA\subseteq C^{\infty}(M,\bR)$ is Fr\'echet-dense;
  \item\label{item:8} $\cA$ separates points and tangent vectors, in the sense that for each non-zero vector $v\in T_xM$ we have $df_x(v)\ne 0$ for some $f\in \cA$;
  \item\label{item:9} $\cA$ separates points and
    \begin{equation*}
      \dim\{df_x,\ f\in \cA\}=n,\quad \forall x\in M. 
    \end{equation*}
  \item\label{item:10} There is some finite subset $f_i\in \cA$, $1\le i\le k$ such that
    \begin{equation*}
      M\ni x\mapsto (f_1(x),\ \cdots,\ f_k(x))\in \bR^k
    \end{equation*}
    is an embedding. 
  \end{enumerate}
\end{theorem}
\begin{proof}
  {\bf \Cref{item:7} $\Rightarrow$ \Cref{item:8}} This is immediate from the fact that certainly, $C^{\infty}(M,\bR)$ itself satisfies the separation conditions in \Cref{item:8}.

  {\bf \Cref{item:8} $\Rightarrow$ \Cref{item:9}} If the differentials $df_x$ spanned a space of dimension $<n$ then they would have to annihilate some non-zero vector in the $n$-dimensional space $T_xM$. 

  {\bf \Cref{item:9} $\Rightarrow$ \Cref{item:10}} We denote by $\overline{\cA}\subseteq C^{\infty}(M,\bR)$ the Fr\'echet closure of $\cA$ and seek to show that the inclusion is an equality.

  For an arbitrary $x\in M$ some $k$-tuple
  \begin{equation*}
    \Psi:=(f_1,\cdots,f_k)\in \cA^k
  \end{equation*}
  has non-zero Jacobian around $x$ and hence is a local $C^\infty$ coordinate system around $x$. We thus have {\it local} embeddability by functions in $\cA$. 

  By the point-separation assumption (and the fact that $\cA$ is unital) the standard Stone-Weierstrass theorem (e.g. \cite[p.122]{rud-fa}) shows that $\cA$ is dense in $C(M,\bR)$ with the supremum norm. It follows that for every inclusion
  \begin{equation*}
    \overline{V}\subset U
  \end{equation*}
  with open $U,V\subseteq M$ and every $\varepsilon>0$ there are functions $\varphi\in\cA$ with
  \begin{equation*}
   \sup_{V}|1-\varphi|<\varepsilon\quad \text{and}\quad \sup_{M\setminus U}|\varphi|<\varepsilon.
  \end{equation*}
  Since there are smooth functions $\theta:\bR\to \bR$ with
  \begin{equation*}
    \theta\circ \varphi|_V\equiv 1\quad\text{and}\quad \theta\circ \varphi|_{M\setminus U}\equiv 0
  \end{equation*}
  and $\theta\circ \varphi\in \overline{\cA}$, the latter algebra contains arbitrary ``bump'' functions: equal to $1$ in any given open set $V$ and $0$ outside any given superset of the closure of $V$.

  Because $M$ is compact, the above local-embeddability conclusion and the existence of bump functions show that we can cover $M$ with finitely many open $U_j$ such that
  \begin{itemize}
  \item functions $f_{j,i}\in \cA$, $1\le i\le k_j$ implement an embedding into $\bR^{k_j}$ of the union of all $U_{j'}$ with
    \begin{equation*}
      \overline{U_j}\cap \overline{U_{j'}}\ne \emptyset;
    \end{equation*}
  \item for $j\ne j'$ such that
    \begin{equation*}
      \overline{U_j}\cap \overline{U_{j'}}= \emptyset;
    \end{equation*}
    we have a function $\psi_{j,j'}$ equal to $1$ on $U_j$ and $0$ on $U_{j'}$. 
  \end{itemize}
  The tuple $\psi_{j,j'}f_{j,i}$ (for all $i$, $j$ and $j'$ as above) will be the desired $f_i$, $1\le i\le k$.
  
  {\bf \Cref{item:10} $\Rightarrow$ \Cref{item:7}} Every smooth function on $M$ will be Fr\'echet-approximable by polynomials in the $f_i$, $1\le i\le k$.
\end{proof}

\begin{remark}\label{re.loc-not-glob}
  The purely local condition on the differentials of $f\in \cA$ would not have sufficed in \Cref{item:8} or \Cref{item:9} of \Cref{th:nachb}: consider for instance the algebra $\cA$ of {\it even} smooth functions on the standard sphere $\bS^n$ (`even' in the sense that $f(x)=f(-x)$). It satisfies the local condition but not the point-separation requirement in the statement above.
\end{remark}

%%%%%%%%%%%%%%%%%%%%%%%%%%%%%%%%%%%%%%%%%%%%%%%%%%%%%%%%%%%%%%%%%%%%%%%%%%%%%
\subsection{Riemannian geometry}\label{subse.ri}

This will be very brief, as good reference sources abound (the reader can consult \cite{doC92,ccl} for instance), though references are much richer for manifolds {\it without} boundary. All of our manifolds are assumed compact, smooth and connected unless specified otherwise.

Given a (compact, smooth, connected) Riemannian manifold $M$ we typically denote by $d$ its {\it geodesic distance}, i.e.
\begin{equation*}
  d(x,y) = \inf_{\gamma}\text{length of $\gamma$},
\end{equation*}
where $\gamma$ ranges over the Lipschitz curves connecting $x$ and $y$ (e.g. \cite[p.2]{grom-struct}). According to the celebrated Hopf-Rinow theorem (e.g. \cite[p.9]{grom-struct}), the compactness (and hence completeness) of $M$ as a metric space under $d$ implies that the distance $d(x,y)$ is always achieved by a {\it minimizing geodesic} (\cite[Definition 1.9]{grom-struct}). In general, we follow standard convention in referring to a curve that locally achieves $d$ as a geodesic (e.g. \cite[Introduction]{abb-cauchy}).

There is a positive $\delta>0$ such that all functions
\begin{equation*}
  d^2_x(-) := d(x,-)^2,\ x\in \overset{\circ}{M} = M\setminus\partial M
\end{equation*}
are smooth on balls in the interior of $M$ of radius $\le \delta$. Indeed, we can simply choose $\delta$ sufficiently small to allow for {\it normal coordinates} in every such ball, where we recall (e.g. \cite[p.145]{ccl}) that a coordinate system on an open neighborhood $U$ of $x\in M$ is normal if the exponential map
\begin{equation*}
  \exp:T_xM\to M
\end{equation*}
maps some open ball around $0\in T_xM$ diffeomorphically onto $U$. The squared distance $d^2_x$ can then be identified, in $\delta$-small neighborhoods  around $x$, with the squared Euclidean distance; clearly, the latter is smooth. 

In order to ``cut off'' large problematic distances where $d^2$ might fail to be smooth we will often work with
\begin{equation}\label{eq:2}
  D(-,-) := \psi\circ d^2(-,-)
\end{equation}
for a smooth ``bump'' function $\psi:\bR\to \bR$ equal to the identity on, say, $\left(-\frac \delta 2,\frac\delta 2\right)$ and vanishing outside $(-\delta,\delta)$ (where $\delta>0$ is chosen sufficiently small, as explained above).

\begin{remark}\label{re.smth_Dy}
  For sufficiently small $\delta$ the map $D:M\times M\to \bR$ is smooth and hence so is
\begin{equation*}
  y\mapsto D_y\in C(M),
\end{equation*}
where smoothness of a map into a possibly-infinite-dimensional Banach space means $C^{\infty}$, as defined for instance on \cite[pp.7-8]{lang-diff}.
\end{remark}

When working with Riemannian manifolds $M$ with boundary $\partial M\ne \emptyset$ we take it for granted that the Riemannian structure can be extended to a closed (i.e. compact, boundary-less) manifold $N\supset M$. Such an extension result follows, for instance, from \cite[Theorem A]{pv}.

%%%%%%%%%%%%%%%%%%%%%%%%%%%%%%%%%%%%%%%%%%%%%%%%%%%%%%%%%%%%%%%%%%%%%%%%%%%%%%%%%%%%%%%%%%%%%%%%%%%%%%%%%%%%%%%%%%
%%%%%%%%%%%%%%%%%%%%%%%%%%%%%%%%%%%%%%%%%%%%%%%%%%%%%%%%%%%%%%%%%%%%%%%%%%%%%%%%%%%%%%%%%%%%%%%%%%%%%%%%%%%%%%%%%%
\section{Smooth actions revisited}\label{se.smth}

In the present section we work with closed manifolds only, i.e. the assumption $\partial M=\emptyset$ is in place throughout.

We refer to \cite{gafa} for a detailed discussion on the natural Fr\'echet topology of $C^\infty(M)$ as well as the space of $\clb$-valued smooth functions $C^\infty(M, \clb)$ for any Banach space $\clb$. Indeed, by the nuclearity of $C^\infty(M)$ as a locally convex space, $C^\infty(M, \clb)$ is the unique topological tensor product of $C^\infty(M)$ and $\clb$ in the category of locally convex spaces. This allows us to define $T \ot {\rm id}$ from $C^\infty(M, \clb)$ for any Fr\'echet continuous linear map $T $ from $ C^\infty(M) $ to $C^\infty(M)$ (or, more generally, to some other locally convex space).  We also recall from \cite{gafa} the space $\Omega^1(M)\equiv \Omega^1(C^\infty(M))$ of smooth one-forms and the space $\Omega^1(M, \clb)$ of smooth $\clb$-valued one-forms, as well as the natural extension of the differential map $d$ to a Fr\'echet continuous map from $C^\infty(M, \clb)$ to $\Omega^1(M, \clb)$. In fact, for $F \in C^\infty(M, \clb)$, the element $dF \in \Omega^1(M, \clb)$ is the unique element satisfying $({\rm id } \ot \xi)(dF(m))=(dF_\xi)(m),$ for every continuous linear functional $\xi$ on $\clb$, where $m \in M,$ $dF(m) \in T^*_m M \ot_{\rm alg} \clb$ and $F_\xi \in C^\infty(M)$ is given by $F_\xi(x):=\xi(F(x))~\forall x \in M.$

The notion of smooth action given below follows \cite{gafa}; we supplement it here with a weaker notion, as follows.

\begin{definition}\label{def.smth}
  An action $\alpha$ of a CQG $\clq$ on $C(M)$ {\it weakly smooth} if
  \begin{equation*}
    \alpha(C^\infty(M)) \subseteq C^\infty(M, \clq).
  \end{equation*}
  $\alpha$ is {\it smooth} if it is weakly smooth and
  \begin{equation*}
    \overline{Sp} \ \alpha(C^{\infty}(M))(1\ot \clq) = C^\infty(M, \clq)
  \end{equation*}
  in the Fr\'echet topology.
 \end{definition}
 \begin{remark}
   In case $\clq=C(G)$ where $G$ is a compact group acting on $M$, say by $\alpha_g: x \mapsto gx$, the smoothness of the induced action $\alpha$ given by $\alpha(f)(x,g)=f(gx)$ on $C(M)$ in the sense of the above definition is equivalent to the smoothness of the map $M \ni x \mapsto gx$ for each $g$. Moreover, in this case smoothness and weak smoothness are equivalent. 
   % argument: If $x \mapsto gx$ is smooth, by averaging we can make $G$ to be closed subgroup of $ISO(M)$ for some Riemannian structure on $M$, hence $G$ becomes a Lie group acting smoothly and isometrically on $M$. Thus, $\alpha_g$ maps $C^\infty(M)$ to itself for every $g$, and is Frechet continuous (as isometry group action factors through $O(n)$-action after we embed $m$ in sufficiently high dimension). By Podles type arguments we get spectral subspaces in $C^\infty$ wrt the Frechet topology consisting of smooth functions, say $\clc_0$, which is Frechet dense in C^\infty$. So, $Sp \alpha(C_0)(1 \ot \clq_0) = \clc_0 \ot \clq_0$, hence Frechet dense in $C^\infty \ot \clq$. This is the def of smooth action in our sense. Conversely, for a smooth action in our sense, $\alpha_g$ maps $C^\infty$ to itself, so
   % \alpha_g(f)=f(g \cdot) is smooth for every smooth $f$, which means $x \mapsto gx$ is a smooth map.
 \end{remark}

 It is proved in \cite[Corollary 3.3]{injective}that for any smooth action $\alpha$, the corresponding reduced action $\alpha_r$ is injective and hence it is implemented by some unitary representation. Moreover, the arguments in \cite{Podles} can be adapted to prove that for a smooth action there is a norm-dense unital $\ast$-subalgebra $\clc_0$ consisting of smooth functions on which $\alpha$ is algebraic. Indeed, this follows from the fact that the spectral projection $P_\pi$ corresponding to any irreducible unitary representation $\pi$ leaves $C^\infty(M)$ invariant and $P_\pi(C^\infty(M))$ is clearly norm-dense in $P_\pi(C(M))$ as $P_\pi$ is a norm-bounded linear operator.

 The norm-density of $\cC_0$ is crucial in \cite[\S 3.2]{final} in producing a Riemannian structure invariant (along with its Laplacian) under the action $\alpha$. We cannot rely on those results directly when the action is only weakly smooth, hence the additional effort below.
 
 We will prove an analogue of the main result of Subsection 3.1 of \cite{final}. To make sense of the statement, recall (e.g. \cite[Definition 3.3]{final})

 \begin{definition}\label{def:riempres}
   Let $M$ be a Riemannian manifold and $\alpha:C(M)\to C(M,\cQ)$ a weakly smooth action. Casting the Riemannian structure as a sesquilinear pairing
   \begin{equation*}
     \langle -,- \rangle:\Omega^1(M)\times \Omega^1(M)\to C^{\infty}(M)
   \end{equation*}
   on 1-forms (complex anti-linear in the first variable), we say that $\alpha$ {\it preserves} (or {\it leaves invariant}) the Riemannian structure if
   \begin{equation*}
     \langle d\alpha(f),d\alpha(g)\rangle_{\cQ}
     =
     \alpha(\langle df,dg\rangle),
     \ \forall f,g\in C^{\infty}(M)
   \end{equation*}
   where
   \begin{itemize}
   \item $d\alpha(f) = (d\otimes \id)\alpha(f)\in \Omega^1(M)\otimes \cQ$ simply applies the differential on the left hand leg of $\alpha(f)$ (this makes sense by weak smoothness);
   \item the $\cQ$-valued inner product is the extension of
     \begin{equation*}
       (\omega_1\otimes x, \omega_2\otimes y)\mapsto \langle \omega_1,\omega_2\rangle\otimes x^*y
     \end{equation*}
     for $\omega_i\in \Omega^1(M)$ and $x,y\in \cQ$. 
   \end{itemize}
 \end{definition}

We can now state
 
 \begin{theorem}\label{th.prsrv}
   \label{metric_pres} Let $\alpha$ be a weakly smooth action of a CQG $\clq$ on a compact Riemannian manifold $M$ such that the corresponding reduced action $\alpha_r$ is injective. Then $\alpha$ preserves some Riemannian metric on $M$. 
 \end{theorem}
\begin{proof}
  If we carefully examine steps of \cite[Theorem 3.7 and Corollary 3.8]{final} it becomes clear that we only need the unitary $U$ which implements the action and the fact that
  \begin{equation*}
    \alpha(C^\infty(M))\subseteq C^\infty(M,\cQ). 
  \end{equation*}
  Adapting those arguments, we can conclude that
  \begin{equation}\label{eq:dfg}
    d\alpha(f)\alpha(g)=\alpha(g) d\alpha(f),\ \forall f,g \in \clc_0.
  \end{equation}
  However, $\clc_0$ is only norm-dense. using that, we get the above identity for all $g \in C(M)$ and all $f \in \clc_0$. Now, we fix $g \in C^\infty(M)$ and use the Leibniz rule (and the commutativity of $\alpha(f)$ with $\alpha(g)$) which gives $\alpha(f)d\alpha(g)=d\alpha(g)\alpha(f)$ for all $g \in C^\infty(M),~f \in \clc_0$. Again using the norm-density of $\clc_0$, we conclude $d\alpha(f)\alpha(g)=\alpha(g) d\alpha(f)$ for all $f,g \in C^\infty(M)$, hence the argument of \cite[Theorem 3.5]{gafa} applies and completes the proof of the present theorem.
\end{proof}

Following \cite[(3)]{gafa}, for $m\in M$ we denote
\begin{equation}\label{eq:qm'}
\cQ'_m:=\text{unital $*$-algebra generated by } \{\alpha(f)(m),\ (X\otimes\id)\alpha(f)(m)\}
\end{equation}
for $f\in C^{\infty}(M)$ and smooth vector fields $X$ on $M$. In the course of the proof of \Cref{th.prsrv} we have shown that

\begin{lemma}\label{le:qm'}
All $\cQ'_m$, $m\in M$ are commutative. 
\end{lemma}
\begin{proof}
  Indeed, this is what \Cref{eq:dfg} says. 
\end{proof}

  % \begin{remark}
 %  Although we don't need it, actually one can prove with a little more effort  that the sesquilinear form coming from $\hat{\cll}$ satisfies the conditions discussed in Remark \ref{riem_nondeg}, hence comes from a Riemannian structure on $M$. For this, one needs a $\ast$-homomorphic  extension of $\epsilon$, say $\epsilon_p$,  on the algebra generated by elements of the form $\alpha(f)(p)$ and $\clc_0$, for any fixed $p \in M$.
  % Such an extension can be defined by  the arguments used in  \cite{injective}. Now, let $p \in M$ and let $f_1, \ldots, f_m \in \clc_0$ (real valued)  be such that they give a set of local coordinates around $p$. As $\hat{\cll} (\clc_0) \subseteq \clc_0$, we have $\hat{\cll}(f_i) \in C^\infty(M)$. To  prove the invertibility of $((k_{\hat{\cll}}(f_i,f_j)(p) ))$, consider $\sum_{ij=1}^m c_ic_j k_{\hat{\cll}}(f_i,f_j)(p)=0$ ($c_i \in \bR$), which  
   %implies (by (\ref{345}) and faithfulness of $h$)  
    %             $\sum_{ij}c_ic_j({\rm id } \ot \kappa)\circ \alpha(k_\cll(f_{i(0)},f_{j(0)}))(x)f_{i(1)}f_{j(1)}=0$. Applying $\epsilon_p$, we get 
     %          $\sum_{ij} c_ic_j< df_{i(0)}\epsilon(f_{i(1)}),
      %               df_{j(0)}\epsilon(f_{j(1)})>_x=
       %               \sum_{ij}c_ic_j<df_i,df_j>_x=0$,
        %               where $< \cdot, \cdot>_x$ is the original  Riemannian inner product 
         %             at $x$. Hence $\sum_i c_i df_i|_x=0$, implying $c_i=0$ $\forall i.$ 
      %\end{remark}
   
We now want to prove the commutativity among higher order partial derivatives. This involves a lift to the cotangent bundle which we can do by following the arguments of \cite[Lemma 3.10]{final} verbatim. However, in order to be able to apply \Cref{metric_pres} to the lift, we must ensure that the corresponding reduced action for the lift is injective. This is equivalent to proving the existence of a faithful positive Borel measure on the sphere bundle of the cotangent space which is preserved by the lifted action. We do this in a few steps.
  
The proof requires some notation. Having fixed an invariant Riemannian metric on $M$, we write $S$ for the unit sphere bundle of the cotangent bundle on $M$:
\begin{equation*}
  \pi:S\subset T^*M\to M. 
\end{equation*}
The typical element of $T^*M$ will be denoted by $(x,\omega)$, where $\omega\in T^*_xM$ is a cotangent vector.

As in \cite[\S 3.3]{final}, for a local chart $U$ on $M$ trivializing $T^*M$ with coordinates $x_1,\cdots,x_n$ we define functions $t^U_j\in C^{\infty}(S)$ by
\begin{equation*}
  t^U_j(x,\omega):=\langle \omega,\omega_j(x)\rangle_x
\end{equation*}
where $\omega_1,\cdots,\omega_n$ is a fixed set of $1$-forms on $M$ orthonormal at every point in $U$ and $\langle-,-\rangle_x$ is the inner product on $T^*_xM$ induced by the Riemannian metric (note the slight abuse of notation: $t^U_j$ depends on the choice of $\omega_j$).

We also define functions $T^U_j\in C^{\infty}(S,\cQ)$ as follows: having extended $\alpha$ to an action $d\alpha$ on the $C^{\infty}(M)$-module of $1$-forms as in \cite[\S 3.2]{gafa} (where that extension is denoted $d\alpha_{(1)}$) and denoting by $\langle\langle-,-\rangle\rangle_x$ the $\cQ$-valued inner product on the Hilbert $\cQ$-module $T^*_xM\otimes \cQ$, we set
\begin{equation*}
  T^U_j(x,\omega):=\langle\langle\omega\otimes 1,d\alpha(\omega_j)(x)\rangle\rangle_x. 
\end{equation*}

As in \cite[\S 3.3]{final}, we construct an action $\beta$ of $\clq$ on $S$ given by
\begin{equation*}
 \beta(f t^U_j)=\alpha(f)T^U_j,\quad  f \in C_c^\infty(U)
\end{equation*}
However, in our case, $\beta$ is only a $C^*$-action, weakly smooth in the sense that 
\begin{equation*}
  \beta(C^{\infty}(S))\subset C^{\infty}(S,\cQ). 
\end{equation*}
Note that we have used the continuity of $\alpha$ in the Fr\'echet topology, which follows from weak smoothness by the closed graph theorem.

\begin{lemma}\label{higher_comm} 
  Assume the hypotheses of \Cref{th.prsrv}.

For any point $x\in M$ and local coordinates $(x_1, \ldots, x_n)$ around $x$, the algebra generated by $\alpha(f)(x), \frac{\partial}{\partial x_{i_1}}\ldots \frac{\partial}{\partial x_{i_k}}\alpha(g)(x)$, where $f,g \in C^\infty(M),$ $k \geq 1$ and $i_j \in \{1,\ldots, n\}$, is commutative.
\end{lemma}
\begin{proof}
  Let $\mu$ be a faithful Borel measure preserved by $\alpha$. Let $\mu_0$ denote the unique $O(n)$ invariant faithful Borel measure (Lebesgue measure) of $S^{n-1}$ and we have a canonical positive, faithful Borel measure on $S$ which is given by the product measure $\mu \times \mu_0$ on any local trivialization. We call this measure $\nu$ and claim that it is preserved by $\beta$.
  
  Choose and fix any locally trivializing neighborhood $U$ and also a function of the form
  \begin{equation*}
    F(e)=f(\pi(e)) P\left(t^U_{j}(e),j=1,\ldots, n\right),
  \end{equation*}
  where $P$ is some polynomial and $f$ has a compact support within $U$. Let $\chi$ be a smooth function with support in $U$ such that $\chi=1$ on the support of $f$.  Now, fix another trivializing neighborhood $V$.  Note that the integral $\int_{\pi^{-1}(V)} G d\nu=\int_{m \in V} d\mu(m) \left( \int_{\pi^{-1}(m)} G_m d\mu_0 \right),$ where $G_m$ is the restriction of $G \in C(S, \clq)$ to the fibre at $m$ which is homeomorphic to $S^{n-1}$. In particular,
  \begin{equation*}
    \int_{\pi^{-1}(V)} \beta(F) d\nu=\int_{m \in V} \alpha(f)(m) \alpha(\chi)(m) \int_{ e \in \pi^{-1}(m)} P\left( T^U_{j}(e),~j=1,\ldots,n \right) d\mu_0.
  \end{equation*}

Recalling the $*$-algebras $\cQ'_m$ defined by \Cref{eq:qm'}, we claim that
\begin{equation}\label{vol_pres}
  \int_{\pi^{-1}(m)} \gamma \left( \alpha(\chi)(m)P( T^U_{j}(e),j=1, \ldots, n) \right) d\mu_0= \int_{\pi^{-1}(m)} \gamma \left( \alpha(\chi)(m)P( t^{U}_{j}(e),j=1, \ldots, n)\right) d\mu_0,
\end{equation}
for any character $\gamma$ on $\cQ'_m$.

Now, it can be proved along the lines of Lemma 3.11 of \cite{gafa} that either $\gamma(\alpha(\chi)(m))$ is zero or we have
\begin{equation*}
  \sum_j \gamma(T^U_j(e))^2=1,\quad \forall e \in \pi^{-1}(m).
\end{equation*}
In case $\gamma(\alpha(\chi)(m))=0$, the equality (\ref{vol_pres}) is immediate. Otherwise, we observe that
\begin{equation*}
  e\equiv (t^U_1(e), \cdots, t^U_n(e)) \mapsto (\gamma(T^U_1(e)),\cdots, \gamma(T^U_n(e)) )
\end{equation*}
gives an isometric map of the fibre $\pi^{-1}(\{ m\}) \cong S^{n-1}$, hence it must be induced by some orthogonal (linear) map restricted to the sphere.  As $\mu_0$ is invariant under any such orthogonal transformation, we have (\ref{vol_pres}). The commutativity of $\cQ'_m$ (\Cref{le:qm'}) then implies the same relation without $\gamma$, i.e.  for all $m \in V,$
\begin{equation*}
  \int_{\pi^{-1}(m)} \alpha(\chi)(m)P(T^U_{j}(e),j=1, \ldots,n ) d\mu_0=\int_{\pi^{-1}(m)} \alpha(\chi)(m)P( t^{U}_{j}(e),j=1, \ldots, n ) d\mu_0.
\end{equation*}
Now, $\int_{\pi^{-1}(m)} P( t^{U}_{j}(e),j=1, \ldots, n )d\mu_0$ does not depend on $m$ and is equal to $C=\int_{\pi^{-1}(m)} \psi(y)d\mu_0(y)$, where $\psi: \pi^{-1}(m) \raro \bR$ given by
\begin{equation*}
  \psi(y\equiv (y_1,\ldots, y_n))=P(y_{i},i=1,\ldots,n). 
\end{equation*}
This gives
\begin{equation*}
  \int_{\pi^{-1}(V)} \beta(F) d\mu=C\int_V \alpha(f) \alpha(\chi) d\mu= C\int_V \alpha(f) d\mu. 
\end{equation*}
As this is true for every locally trivializing $V$ we get by a partition of unity argument $\int_{S} \beta(F)d\nu=C \int_M \alpha(f)d\mu= C (\int_M f d\mu)1_\clq=(\int F d\nu) 1_\clq$, as $\int F d\nu$ is clearly equal to $C\int_M f d\mu$.

Thus, the lifted action $\beta$ on $S$ remains weakly smooth and $\beta_r$ is injective. We can now follow the iterative arguments of \cite{final} to complete the proof of higher order commutativity.
\end{proof}

The proof of the main theorem of \cite{final} now goes through verbatim to give is the following:

\begin{theorem}\label{main}
  Let $\alpha$ be a weakly smooth faithful action of a CQG $\clq$ on a compact connected smooth manifold $M$. Then $\clq$ must be classical, i.e. isomorphic with $C(G)$ for a compact group $G$ acting smoothly on $M$.
\end{theorem}
\begin{proof}
  By passing to the reduced action $\alpha_r$ we can assume without loss of generality that the action preserves some faithful positive Borel measure on $M$.

  Note that the isometry condition, i.e. commutation with the Laplacian, is not used up to \cite[\S 3]{gafa}, so all those results are valid for a weakly smooth actions. Following the arguments of \cite{gafa} we can prove that a weakly smooth action commutes with the Laplacian on $\clc_0$, but absent Fr{\'e}chet density, it will not be a core for the Laplacian and commutation does not extend to $C^\infty(M)$. Nevertheless, \Cref{higher_comm} already proves the commutativity of higher-order partial derivatives, bypassing the arguments of \cite[\S 4]{gafa} (which used the isometry condition). The proof of \cite[Theorem 5.3]{gafa} then carries through more or less verbatim, as we proceed to sketch briefly.

  Given the smooth action $\alpha$ of $\clq$ on $M$, we choose a Riemannian metric by Corollary \ref{metric_pres} which is preserved by the action.  This implies the commutativity of $\clq'_x$. Using this, we can proceed along the lines of \cite{gafa} to lift the given action to $O(M)$.  Now, by Lemma \ref{higher_comm}, we do have the commutativity of partial derivatives of all orders for the lifted action $\Phi$ needed in steps (i) and (iv) of the proof of Theorem 5.3 of \cite{gafa} and the rest of the arguments of Theorem 5.3 of \cite{gafa} will go through.
\end{proof}

%%%%%%%%%%%%%%%%%%%%%%%%%%%%%%%%%%%%%%%%%%%%%%%%%%%%%%%%%%%%%%%%%%%%%%%%%%%%%%%%%%%%%%%%%%%%%%%%%%%%%%%%%%%%%%%%%%
%%%%%%%%%%%%%%%%%%%%%%%%%%%%%%%%%%%%%%%%%%%%%%%%%%%%%%%%%%%%%%%%%%%%%%%%%%%%%%%%%%%%%%%%%%%%%%%%%%%%%%%%%%%%%%%%%%
\section{Quantum isometry groups: existence and rigidity for closed manifolds}\label{se.main}

% % %%%%%%%%%%%%%%%%%%%%%%%%%%%%%%%%%%%%%%%%%%%%%%%%%%%%%%%%%%%%%%%%%%%%%%%%%%%%%
% % \subsection{Closed manifolds}\label{subse.nbdry}
% %

Let $M$ be a smooth compact Riemannian manifold without boundary and $d$ its geodesic distance, as before. If $\alpha$ is an isometric CQG action on $(M,d)$ then it automatically preserves all functions of the form
\begin{equation*}\label{eq:13}
  \psi\circ d\in C(M\times M)
\end{equation*}
for continuous $\psi:\bR\to \bR$. In particular, it will preserve the function $D(-,-)$ defined by \Cref{eq:2}. We write $D_x$, $x\in M$ for the function $D(x,-)$.

\Cref{le.d_x_density} below will implicitly make use of the following observation.

\begin{lemma}\label{le.d_x_density}
  For a compact connected Riemannian manifold $M$ without boundary the algebra generated by $\{ D_x:~x \in M \}$ is Fr\'echet-dense in $C^\infty(M)$.
\end{lemma}
\begin{proof}
  We will apply \Cref{th:nachb} by verifying, say, condition \Cref{item:9} in that statement. Since an appeal to Stone-Weierstrass quickly shows that the algebra in question is norm-dense, only the local condition needs verification. That is, if $\clc$ is the linear span of functions of the form $D_x, x \in M$, we have to show that for any point $y \in M$ the space $\{ df|_y, f \in \clc\}$ is $n$-dimensional (where $n=\dim M$). We thus focus on proving this full-dimension claim.

  Suppose there is some $y$ for which
  \begin{equation*}
    \dim \{df|_y,\ f\in \cC\}<n.
  \end{equation*}
  Then there is some unit tangent vector $v\in T_yM$ for which $(df_y, v)=0$ for all $f=D_x$. Now consider the arc-length-parametrized geodesic starting at $y$ with velocity $v$ and let $x$ be a point on it, sufficiently close to $y$ to ensure that some normal coordinate neighborhood \cite[p.145]{ccl} $U$ of $x$ contains $y$ and that
  \begin{equation*}
    D(x,-) = d(x,-)^2
  \end{equation*}
  throughout $U$. 

  If $\exp:T_yM\to M$ is the exponential map, we now have
  \begin{equation*}
    D_x(\exp(tv)) = (d(x,y)-t)^2D(x,y),
  \end{equation*}
  whose derivative at $t=0$ clearly does not vanish. This gives the desired contradiction and finishes the proof.
\end{proof}

\begin{theorem}\label{th.nbdry}
  Let $M$ be a Riemannian closed connected smooth manifold and $d$ the corresponding geodesic metric. Then $QISO(M, d)$ exists and coincides with $C(ISO(M))$ where $ISO(M)$ is the group of Riemannian isometries.
\end{theorem}
\begin{proof}
  We denote by $D$ a function $\psi\circ d^2$ as in \Cref{eq:2}, for a bump function $\psi:\bR\to \bR$ equal to $\id$ around $0\in \bR$ and vanishing outside a sufficiently small neighborhood of $0$.

  We know from \cite[Proposition 3.10]{Chi15} that every reduced isometric action is injective, so \Cref{th.prsrv} applies. It is thus enough to prove that any CQG isometric action $\alpha$ on $C(M)$ is weakly smooth.

To see this, recall from \Cref{def.isometric} that the isometric property of the action reads
\begin{equation*}
\alpha(D_x)(y)=\kappa(\alpha(D_y)(x)).  
\end{equation*}
Fixing $x$, we now examine the function
\begin{equation*}
  y\mapsto \alpha(D_x)(y)=\kappa(\alpha(D_y)(x)). 
\end{equation*}
It is the composition between the smooth function
\begin{equation*}
  M\ni y\mapsto D_y\in C(M)
\end{equation*}
(see \Cref{re.smth_Dy}) and the $C^*$-algebra morphism
\begin{equation*}
  C(M)\ni f\mapsto \kappa(\alpha(f)(x)),
\end{equation*}
and hence is itself smooth.

%  Indeed, in the notation of \cite{metric}, $\alpha(D_x)(y)=\kappa(\alpha(D_y)(x)).$ Fix $x$ and note that $M \ni y \mapsto D_y \in C^\infty(M)$ is continuous wrt the Fr\'echet topology on $C^\infty(M)$. As $\alpha$ is Fr\'echet continuous and $\kappa$ is norm bounded, we have $y \mapsto \kappa(\alpha(D_y)(x))$ is Fr\'echet continuous too. That is, $y \mapsto \alpha(D_x)(y)$ is Fr\'echet continuous, hence a smooth map.

By \Cref{th:nachb}, we can find finitely many polynomials $\xi_i$ in the functions $D_x$ such that
\begin{equation*}
  y \mapsto (\xi_1(y) , \ldots ,\xi_k(y))
\end{equation*}
is a smooth embedding of $M$ as a submanifold in $\bR^k$. Every smooth function $f$ on $M$ can be written as $f=\tilde{f}(\xi_1, \ldots, \xi_k)$ for some function $\tilde{f}$ of $k$ real variables with compact support in some open neighborhood of $M$, so all in all we obtain
  \begin{equation*}
    \alpha(f)=\tilde{f}(\alpha(\xi_1), \ldots, \alpha(\xi_k))
  \end{equation*}
Since we have just argued that $\alpha(\xi_i) \in C(M,\cQ)$ are smooth, so is $\alpha(f)$, finishing the proof. 
\end{proof}

%%%%%%%%%%%%%%%%%%%%%%%%%%%%%%%%%%%%%%%%%%%%%%%%%%%%%%%%%%%%%%%%%%%%%%%%%%%%%
%%%%%%%%%%%%%%%%%%%%%%%%%%%%%%%%%%%%%%%%%%%%%%%%%%%%%%%%%%%%%%%%%%%%%%%%%%%%%
\section{Manifolds with boundary}\label{subse.bdry}

As the title suggests, we will now extend the quantum rigidity result in \Cref{th.nbdry} to the case when $\partial M\ne \emptyset$. To that end, throughout the present subsection $M$ denotes a compact Riemannian manifold with boundary

Consider an action of a compact quantum group $\cQ$ on $M$, with $\cC=C(M)$:
\begin{equation}\label{eq:6}
  \alpha:\cC\to \cC\otimes \cQ. 
\end{equation}
For the actions we are interested in (isometric with respect to the geodesic distance of a Riemannian structure), it will be crucial to know that they leave the boundary invariant, in a sense to be made precise below.

%%%%%%%%%%%%%%%%%%%%%%%%%%%%%%%%%%%%%%%%%%%%%%%%%%%%%%%%%%%%%%%%%%%%%%%%%%%%%
\subsection{Invariant subspaces}\label{subse:inv}

\begin{definition}\label{def.prsv}
Let $Z\subseteq X$ be an inclusion of compact Hausdorff spaces and \Cref{eq:6} an action of a compact quantum group on $\cC=C(X)$. We say that $\alpha$ {\it preserves $Z$} or that $Z$ is {\it $\alpha$-invariant} (or just plain {\it invariant} when $\alpha$ is understood) if we have a factorization 
\begin{equation}\label{eq:7}
  \begin{tikzpicture}[auto,baseline=(current  bounding  box.center)]
    \path[anchor=base] (0,0) node (1) {$\cC$} +(3,.5) node (2) {$\cC\otimes \cQ$} +(6,0) node (3) {$C(Z)\otimes \cQ$} +(3,-.5) node (4) {$C(Z)$};
    \draw[->] (1) to[bend left=6] node[pos=.5,auto] {$\scriptstyle \alpha$} (2);
    \draw[->] (2) to[bend left=6] node[pos=.5,auto] {$\scriptstyle \pi\otimes\id$} (3);
    \draw[->] (1) to[bend right=6] node[pos=.5,auto,swap] {$\scriptstyle \pi$} (4);
    \draw[->] (4) to[bend right=6] node[pos=.5,auto,swap] {$\scriptstyle \beta$} (3);        
  \end{tikzpicture}
\end{equation}
where
\begin{equation*}
  \pi:\cC=C(X)\to C(Z)
\end{equation*}
is restriction.
\end{definition}

Assuming such a factorization does exist, the lower right hand arrow $\beta$ will automatically be an action. Denoting by $J=J_Z\subset C(X)$ the ideal of functions vanishing along $Z$, $\alpha$ restricts to a map
\begin{equation*}
  J\to J\otimes \cQ = C_0(U,\cQ)
\end{equation*}
where $U:=X-Z$ and $C_0$ means functions vanishing at infinity on the locally compact space $U$. We will also need

\begin{definition}\label{def:strprsv}
  In the context of \Cref{def.prsv} $Z$ is {\it strongly $\alpha$-invariant} (or $\alpha$ {\it preserves $Z$ strongly}) if the restriction $J_Z\to J_Z\otimes \cQ$ of $\alpha$ satisfies the density condition \Cref{item:12} in \Cref{def:act} for an action (we say in short that $\alpha$ induces an action of $\cQ$ on $U=X-Z$):
  \begin{equation}\label{eq:zxu}
    \overline{Sp} \ \alpha(J_Z)(1\ot \cQ) = C_0(U,\cQ). 
  \end{equation}
\end{definition}

\begin{remark}\label{re:notauto}
  We do not know whether \Cref{def:strprsv} is redundant, i.e. whether strong preservation is automatic given preservation. In fact, if $\alpha:C(X)\to C(X,\cQ)$ is not injective (a possibility we cannot rule out at the moment) and $Z$ is the spectrum of the proper quotient $\alpha(C(X))$ of $C(X)$, then
  \begin{itemize}
  \item $Z$ is preserved by $\alpha$, but nevertheless
  \item by construction, every element of $J_Z$ is annihilated by $\alpha$. 
  \end{itemize}
\end{remark}

Now let $\alpha$ be an action of $\cQ$ on $X$ as before, and $Z\subseteq X$ an $\alpha$-invariant subspace. We denote the dense $*$-subalgebras resulting from this as recalled in \Cref{subse:cqg} by `$0$' subscripts, as in $\cQ_0$, $C(X)_0$, etc. Our first observation on strong invariance is

\begin{lemma}\label{le:isdense}
  If $Z\subseteq X$ is strongly $\alpha$-invariant then the non-unital $*$-algebra
  \begin{equation*}
    C(X)_0\cap J_Z
  \end{equation*}
  of elements of $C(X)_0$ vanishing along $Z$ is dense in $J_Z$. 
\end{lemma}
\begin{proof}
  One simply imitates the usual proof that $C(X)_0\subseteq C(X)$ is dense; see e.g. \cite[Theorem 1.5]{Podles}. Alternatively, we can simply {\it apply} that density result to the $\cQ$-action on the one-point compactification of $U$ induced by $\alpha$; that the map
  \begin{equation*}
    C_0(U)^+\to C_0(U)^+\otimes \cQ
  \end{equation*}
  (where `$+$' superscripts denote unitizations) is indeed an action requires precisely the density condition strong invariance imposes.
\end{proof}

\Cref{le:isdense} will come in handy in the context of ``gluing'' actions along a common subspace of two spaces. The setup is as follows. Let $X_i$, $i=1,2$ be compact Hausdorff spaces equipped with actions
\begin{align*}
  \alpha_i &:C(X_i)\to C(X_i)\otimes \cQ
\end{align*}
by a quantum group $\cQ$ and
\begin{align*}
  \iota_i:Z\to X_i 
\end{align*}
embeddings of compact spaces. We write $X:=X_1\cup_ZX_2$ for the resulting space obtained by gluing $X_i$ along $Z$ via the embeddings $\iota_i$ (though by a slight abuse of notation said embeddings are absent from the notation).

Setting $Y:=X_1\sqcup X_2$, we have a product action
\begin{equation}\label{eq:10}
 \begin{tikzpicture}[auto,baseline=(current  bounding  box.center)]
  \path[anchor=base] (0,0) node (1) {$C(Y)$} +(2,1) node (2) {$C(X_1)\times C(X_2)$} +(8,1) node (3) {$(C(X_1)\otimes \cQ)\times (C(X_2)\otimes \cQ)$} +(12,0) node (4) {$C(Y)\otimes \cQ$.};
  \draw[->] (1) to[bend left=6] node[pos=.5,auto] {$\scriptstyle \cong$} (2);
  \draw[->] (2) to[bend left=6] node[pos=.5,auto] {$\scriptstyle \alpha_1\times\alpha_2$} (3);
  \draw[->] (3) to[bend left=6] node[pos=.5,auto] {$\scriptstyle \cong$} (4);
  \draw[->] (1) to[bend right=6] node[pos=.5,auto,swap] {$\scriptstyle \beta$} (4);    
 \end{tikzpicture}
\end{equation}
Now assume furthermore that $\alpha_i$
\begin{itemize}
\item preserve the subspaces
\begin{equation*}
  \iota_i(Z)\subseteq X_i,\ i=1,2
\end{equation*}
in the sense of \Cref{def.prsv}, i.e. for any $f\in C(X_i)$ the restriction of
\begin{equation*}
  \alpha_i(f)\in C(X_i,\cQ)
\end{equation*}
to $\iota_i(Z)\subseteq X_i$ depends only on the restriction of $f$ to the same subspace.
\item agree on $Z$, in the sense that the actions on $Z$ induced by $\alpha_i$ upon identifying $Z\cong \iota_i(Z)$ coincide. 
\end{itemize}
In this case, if $f_i\in C(X_i)$ have equal restrictions to $Z$ via $\iota_i$ then similarly,
\begin{equation*}
    \alpha_i(f_i)\in C(X_i,\cQ)
\end{equation*}
have equal restrictions to $Z$. But this simply means that with $\beta$ defined as in \Cref{eq:10},
\begin{equation*}
  \beta(f_1,f_2)\in C(X\otimes \cQ)\cong C(X)\otimes \cQ. 
\end{equation*}
Since this holds for arbitrary $(f_1,f_2)\in C(X)$ we have

\begin{proposition}\label{le.glue}
  If actions $\alpha_i$ of $\cQ$ on compact spaces $X_i$ strongly preserve a common subspace $Z\subseteq X_i$ on which they agree, we obtain a natural action $\alpha$ of $\cQ$ on the connected sum $X=X_1\cup_Z X_2$.

  If at least one of the actions $\alpha_i$ is faithful then so is $\alpha$ and if $(\alpha_i)_r$ are injective then so is $\alpha_r$. 
\end{proposition}
\begin{proof}
  The proof of the existence of $\alpha$ is essentially contained in the discussion preceding the statement, with the possible caveat that we have not argued that the density condition \Cref{item:12} in the definition of an action holds:
  \begin{equation}\label{eq:wantdense}
    \overline{Sp} \ \alpha(C(X))(1\ot \cQ) = C(X,\cQ). 
  \end{equation}
  We can see this by working at the purely algebraic level, with the dense subalgebras 
  \begin{equation*}
    C(X)_0\subset C(X) 
  \end{equation*}
  and similarly for the spaces $X_i$ and $Z$ (but not $X$ yet, as we do not know at this stage that $\cQ$ acts on it), and with $\cQ_0\subset \cQ$ in place of $\cQ$. The $\cQ$-equivariant embeddings $Z\subseteq X_i$ induce surjections
  \begin{equation}\label{eq:xitoz}
    C(X_i)_0\to C(Z)_0,
  \end{equation}
  giving us a coaction of the Hopf algebra $\cQ_0$ on the pullback $C(X)_0$ of these surjections in the category of $*$-algebras:
  \begin{equation*}
\begin{tikzpicture}[auto,baseline=(current  bounding  box.center)]
\path[anchor=base] 
(0,0) node (l) {$C(X_1)_0$}
+(2,.5) node (u) {$C(X)_0$}
+(2,-.5) node (d) {$C(Z)_0$}
+(4,0) node (r) {$C(X_2)_0$}
;
\draw[<-] (l) to[bend left=6] node[pos=.5,auto] {$\scriptstyle $} (u);
\draw[->] (u) to[bend left=6] node[pos=.5,auto] {$\scriptstyle $} (r);
\draw[->] (l) to[bend right=6] node[pos=.5,auto,swap] {$\scriptstyle $} (d);
\draw[<-] (d) to[bend right=6] node[pos=.5,auto,swap] {$\scriptstyle $} (r);
\end{tikzpicture}
\end{equation*}

{\bf Claim: The pullback $C(X)_0$ is dense in $C(X)$.} This is where {\it strong} preservation will be needed. Consider an arbitrary element of $C(X)$, i.e. a pair of continuous functions $f_i$ on $X_i$ (respectively) agreeing on $Z$. Approximate $f_2$ arbitrarily well (within $\varepsilon>0$, say) with an element
\begin{equation*}
  f_{2,app}\in C(X_2)_0
\end{equation*}
and restrict the latter to $f_{Z,app}\in C(Z)_0$. In turn, {\it that} function lifts to some element $g_{1,app}\in C(X_1)_0$. On the other hand, because $f_{Z,app}$ and the common restriction
\begin{equation*}
  f_Z:=f_1|_Z = f_2|_Z
\end{equation*}
are within $\varepsilon$, their difference $f_{Z,app}-f_Z$ lifts to a function $h_1$ on $X_1$ of norm $<2\varepsilon$. The sum
\begin{equation*}
  f_1 + h_1\in C(X_1)
\end{equation*}
is $2\varepsilon$-close to $f_1$ and restricts to $f_{Z,app}\in C(Z)_0$ on $Z$. So does $g_{1,app}\in C(X_1)_0$, so the difference
\begin{equation}\label{eq:3dif}
  f_1+h_1-g_{1,app}
\end{equation}
vanishes on $Z$. But then, by the strong preservation of $Z\subseteq X_1$ by $\alpha_1$, \Cref{eq:3dif} is $\varepsilon$-close to some element
\begin{equation*}
  d_{1,app}\in C(X_1)_0\cap C_0(X_1-Z)
\end{equation*}
(i.e. in the dense $*$-subalgebra $C(X_1)_0\subseteq C(X_1)$ and vanishing on $Z$). All in all, $f_1$ is within $2\varepsilon$ of $f_1+h_1$, which in turn is within $2\varepsilon$ of
\begin{equation}\label{eq:g+d}
  g_{1,app}+d_{1,app}\in C(X_1)_0. 
\end{equation}
Furthermore, because $d_{1,app}$ vanishes on $Z$, \Cref{eq:g+d} agrees with $f_{2,app}$ along $Z$. This finishes the proof of the claim: the arbitrary pair of functions $f_i\in X_i$ agreeing on $Z$ has been $4\varepsilon$-approximated by a pair of functions $f_{i,app}\in C(X_i)_0$ agreeing on $Z$. 
  
  With the claim proven \Cref{eq:wantdense} follows, since in general, at the level of Hopf algebra coactions
  \begin{equation*}
    \cA\ni a\longmapsto a_0\otimes a_1\in \cA\otimes_{\rm alg}\cQ_0
  \end{equation*}
  the bijectivity of
  \begin{equation*}
    \cA\otimes_{\rm alg}\cQ_0\ni a\otimes b\longmapsto a_0\otimes a_1b\in \cA\otimes_{\rm alg}\cQ_0
  \end{equation*}
  follows from the existence of the antipode $\kappa$ on $\cQ_0$: the inverse is simply
  \begin{equation*}
    \cA\otimes_{\rm alg}\cQ_0\ni a\otimes b\longmapsto a_0\otimes \kappa(a_1)b\in \cA\otimes_{\rm alg}\cQ_0. 
  \end{equation*}

  For the faithfulness claim, note that for points $x_i\in X_i$ we have
  \begin{equation*}
    \cQ_{x_i}^{\alpha_i} = \cQ_{x_i}^{\alpha}\subseteq \cQ. 
  \end{equation*}
  Since we are assuming these algebras generate $\cQ$ as $x_i$ ranges over $X_i$ for at least one of the indices $i=1,2$, the slice algebras $\cQ_{x}^{\alpha}$ do indeed generate $\cQ$ as $x\in X=X_1\cup_Z X_2$.

  Finally, suppose $(\alpha_i)_r$ are injective. Since every non-zero function $f\in C(X)$ restricts to a non-zero function on at least one $X_i$ and both $X_i$ are preserved by $\alpha_r$ which induces back the actions
  \begin{equation*}
    (\alpha_i)_r:C(X_i)\to C(X_i,\cQ),
  \end{equation*}
  we have $\alpha_r(f)\ne 0$, as desired. 
\end{proof}

\begin{remark}
  Although we do not use this, note that in fact the proof of \Cref{le.glue} shows that it is enough to assume {\it one} of the actions $\alpha_i$ preserves $Z$ strongly.
\end{remark}

%%%%%%%%%%%%%%%%%%%%%%%%%%%%%%%%%%%%%%%%%%%%%%%%%%%%%%%%%%%%%%%%%%%%%%%%%%%%%
\subsection{Back to manifolds}\label{subse:back}

We begin with precisely the boundary-invariance result alluded to at the beginning of \Cref{subse.bdry}.

\begin{proposition}\label{pr.bdry-inv}
  Let $M$ be a compact Riemannian manifold and $d$ its geodesic distance. Then, every reduced isometric CQG action on $(M,d)$ preserves the boundary.
\end{proposition}
\begin{proof}
  We have to argue that if $x\in \partial M$ then the entire $\alpha$-orbit (\cite[\S 3]{chirvasitu}) of $x$ is contained in the boundary. To see this we assume otherwise and derive a contradiction.

  Suppose $y\in \overset{\circ}M=M\setminus \partial M$ is a point in the orbit of $x$ and $\varphi$ is a state on $\cQ$ with $x\triangleleft \varphi = y$. We also denote by $x'\in \overset{\circ}M$ a point placed a small distance $r$ away from $x$, connected to the latter by a geodesic arc $\gamma$ orthogonal to the boundary at $x$.

  The probability measure $x'\triangleleft \varphi$ is supported on the sphere $S(y,r)$ of radius $r$ around $y=x\triangleleft\varphi$ (e.g. by \cite[Theorem 3.1]{Chi15}), and we may assume $r>0$ is small enough that that sphere is entirely within the interior of $M$. Let
  \begin{equation}\label{eq:3}
    y'\in\mathrm{supp}~(x'\triangleleft\varphi )
  \end{equation}
  and denote by $y''\in S(y,r)$ the antipode opposite $y'$, so that
  \begin{equation}\label{eq:4}
    d(y,y') = d(y,y'') = \frac {d(y',y'')}{2} = r. 
  \end{equation}
  Now denote $\overline{\varphi}=\varphi\circ\kappa$. It follows from \Cref{eq:3} and \cite[Proposition 3.1]{chirvasitu} that
  \begin{equation}\label{eq:5}
    x'\in\mathrm{supp}~\left(y'\triangleleft\overline{\varphi} \right)
  \end{equation}
  (and note that we also have $y\triangleleft\overline{\varphi}=x$, by \cite[Corollary 3.2]{chirvasitu}). All in all, $\overline{\varphi}$ maps
  \begin{itemize}
  \item $y\in \overset{\circ} M$ to $x\in \partial M$;
  \item $y'\in S(y,r)$ to a measure whose support contains $x'$ and is contained in $S(x,r)$.
  \item $y''\in S(y,r)$ to a measure supported on the same sphere $S(x,r)$, by \Cref{eq:4}. 
  \end{itemize}
  The last equality in \Cref{eq:4} and \cite[Theorem 3.1]{Chi15} also show that there is a probability measure on $M\times M$, supported on
  \begin{equation*}
    \{(p,q)\in M\times M\ |\ d(p,q)=2r\},
  \end{equation*}
  whose pushforwards through the two projections are $y'\triangleleft\overline{\varphi}$ and $y''\triangleleft\overline{\varphi}$. \Cref{eq:5} now implies that there is some
  \begin{equation*}
    x''\in \mathrm{supp}\left(y''\triangleleft\overline{\varphi}\right)\subseteq S(x,r)
  \end{equation*}
  with $d(x',x'')=2r$. This, however, contradicts the choice of $x'$: since the geodesic arc $\gamma$ connecting $x$ and $x'$ has length $r$ and is orthogonal to $\partial M$ at $x$, the antipode of $S(x,r)\subset N$ opposite $x'$ (for an extension $N\supset M$ as in \Cref{subse.ri}) is not contained in $M$.
\end{proof}

Denote 
\begin{align*}
  \partial_rM        &:=\{x\in M\ |\ d(x,\partial M)=r\}\\
  \partial_{\le r}M  &:=\{x\in M\ |\ d(x,\partial M)\le r\}\label{eq:partials}\numberthis\\
  \partial_{< r}M  &:=\{x\in M\ |\ d(x,\partial M)< r\}
\end{align*}
and similarly for `$\ge $', `$>$', etc. For $r\le s$ set
\begin{equation*}
  \partial_{s\la r}M = \{x\in \partial_rM |\ \exists y\in \partial_s M\text{ such that }d(x,y)=s-r\}.
\end{equation*}

The following result is now an immediate consequence of \Cref{pr.bdry-inv}.

\begin{corollary}\label{cor.prsv}
  Under the hypotheses of \Cref{pr.bdry-inv} the action $\alpha$ preserves the compact sets $\partial_r M$, $\partial_{\ge r}M$ and $\partial_{s\la r}M$ for all real numbers $0\le r\le s$.  \qedhere
\end{corollary}

This ensures that for each $r\ge 0$ we have an action $\beta$ as in \Cref{eq:7} for $X=\partial_r M$. We will be interested in the following choices of $r$.

\begin{definition}\label{def.tame}
  Let $M$ be a compact Riemannian manifold with boundary. A positive real $r$ is {\it tame} if it is sufficiently small so that $\partial_r M$ is contained in a collar neighborhood of $\partial M$ with a system of coordinates {\it adapted} to the boundary: $x_n$ is distance from $\partial M$ and $x_i$, $1\le i\le n-1$ are coordinates on the boundary extended as constant along geodesic arcs orthogonal to $\partial M$.
\end{definition}

If we knew that the resulting action $\beta$ is faithful we could conclude that the quantum group $\cQ$ is classical by a slight adaptation of \Cref{th.nbdry}. Though this is not quite the strategy we adopt for generalizing \Cref{th.nbdry} to \Cref{th.rig-bdry} below, we nevertheless prove that $\beta$ is faithful for whatever independent interest that result might hold and also because the requisite techniques will be useful later.

According to \Cref{def.CQG_action} an action $\alpha$ is faithful if $\cQ$ is generated as a $C^*$-algebra by the subalgebras
\begin{equation}\label{eq:8}
  \cQ_x = \cQ^{\alpha}_x:=\mathrm{Im} ({\rm ev}_x \ot {\rm id})\circ\alpha. 
\end{equation}

Note that this differs from the algebra denoted by $\cQ_x$ in \cite{final}; indeed, in the present paper the latter algebra would be denoted by $\cQ'_x$ instead.

We need the following notion.

\begin{definition}\label{def.att}
  Consider an action $\alpha$ as in \Cref{eq:6} and $x,y\in M$ two points. We say that $y$ is {\it $\alpha$-attached} to $x$ (or just {\it attached} when the action is understood) if for states $\varphi$ and $\psi$ on $S$ we have
  \begin{equation*}
    x\triangleleft\varphi = x\triangleleft\psi\quad \Rightarrow\quad y\triangleleft\varphi = y\triangleleft\psi. 
  \end{equation*}
\end{definition}

The concept is relevant to faithfulness due to the following result proved in passing in the course of the proof of \cite[Proposition 4.4]{chirvasitu}.

\begin{proposition}\label{pr.att}
  Let $(M,d)$ be a compact metric space, $\alpha$ an isometric action of a compact quantum group $\cQ$ on $M$ and $x,y\in M$. If $y$ is $\alpha$-attached to $x$ then $\cQ_y\subseteq \cQ_x$.  \qedhere
\end{proposition}

Going back to the situation at hand, consider the action $\beta$ on $X=\partial_r M$ resulting from $\alpha$ as in \Cref{eq:7}. For $x\in \partial_r M$ the subalgebra $\cQ_x^{\beta}$ defined as in \Cref{eq:8} coincides with $\cQ_x^{\alpha}$. On the other hand, \Cref{pr.att} shows that $\cQ^{\alpha}_y$ is contained in $\cQ^{\alpha}_x$ whenever $y$ is attached to $x$. Since we know (from the faithfulness of $\alpha$) that
\begin{equation*}
  \cQ^{\alpha}_y,\ y\in M
\end{equation*}
generate $\cQ$, we will have shown that $\beta$ is indeed faithful provided we prove

\begin{proposition}\label{le.is-att}
  Let $M$, $\alpha$, etc. be as above, with the additional assumption that every component of $M$ has non-empty boundary. For sufficiently small $r>0$ every point in $M$ is $\alpha$-attached to some $x\in \partial_r M$.
\end{proposition}

We will prove this in a few stages. First, we have

\begin{lemma}\label{le.uniq-rs}
  Let $0<r$. There is some $\varepsilon>0$, depending only on the Riemannian manifold $M$, with the following property: 

  For every $s>r$ with $s-r\le \varepsilon$ and $x\in \partial_{s\la r}M$ the set
  \begin{equation}\label{eq:9}
    \{y\in \partial_s M\ |\ d(x,y)=s-r\}
  \end{equation}
  is a singleton. 
\end{lemma}
\begin{proof}
  Choose $0<\varepsilon<r$ smaller than the injectivity radius of $M$ at every point
  \begin{equation*}
    p\in M,\ d(p,\partial M)\ge r.
  \end{equation*}  
  The very definition of $\partial_{s\la r}M$ says that the set in question is non-empty, so we have to prove that the set \Cref{eq:9} cannot contain distinct points $y\ne y'$.

  Indeed, two such points would entail the existence of two distinct geodesic arcs
  \begin{equation*}
    \gamma:x\to y,\quad \gamma':x\to y'
  \end{equation*}
  of length $s-r$. They cannot both prolong a geodesic arc $\eta$ of length $r$ connecting $x$ to $\partial M$, so one of the concatenations
  \begin{equation*}
    \eta\cdot \gamma,\quad \eta\cdot \gamma'
  \end{equation*}
  is not a geodesic. But both curves have length $r+s-r=s$, meaning that one of the two points $y,y'\in \partial_s M$ can be connected to $\partial M$ by a curve of length $<s$. This contradiction finishes the proof.
\end{proof}

Now let $r>0$. According to \Cref{le.uniq-rs}, for every $s>r$ sufficiently close to $r$ there is a well-defined map
\begin{equation}\label{eq:psis}
  \psi_{s\la r}:\partial_{s\la r} M\to \partial_s M\quad
  \text{such that}\quad
  d(x,\psi_{s\la r}(x)) = s-r.
\end{equation}
Furthermore, uniqueness implies
\begin{itemize}
\item the transitivity of $\psi$:
\begin{equation*}
  \psi_{s_2\la r} = \psi_{s_2\la s_1}\circ \psi_{s_1\la r}
\end{equation*}
for $s_2>s_1>r$ sufficiently close to $r$;
\item the continuity of each $\psi_{s\la r}$: if $x_n\to x$ is a convergent sequence in $\partial_{s\la r}M$, then by the continuity of the distance function the limit of every convergent subsequence of $(\psi_{s\la r}(x_n))_n$ is at distance $s-r$ from $x=\lim_n x_n$, so that limit must be $x$. It follows that $(\psi_{s\la r}(x_n))_n$ must be convergent to this common limit point, since it is a sequence in a compact metric space. 
\end{itemize}

By transitivity, we can define $\psi_{s\la r}$ for arbitrary
\begin{equation*}
  r\le s<\max_p d(p,\partial M). 
\end{equation*}

\pf{le.is-att}
\begin{le.is-att}
  Of course, it suffices to argue that points $y$ in the interior $\overset{\circ} M$ are attached to points on the boundary.

  Let $\ell=d(y,\partial M)$ and $\gamma$ a shortest geodesic, parametrized by arclength, connecting some point $\gamma(0)=x\in \partial M$ to $\gamma(\ell)=y$ (the existence of such a geodesic requires our assumption that all connected components have boundary). Note that we have
  \begin{equation*}
    \gamma(t)\in \partial_t M,\ \forall t\in [0,\ell]. 
  \end{equation*}
  Now let $\varphi$ be a state on $\cQ$ and $r>0$ tame for $M$ in the sense of \Cref{def.tame}. As in the proof of \Cref{pr.bdry-inv}, we can conclude from \cite[Theorem 3.1]{Chi15} that for every $\ell\ge s>r$ the measures
  \begin{equation*}
    \gamma(r)\triangleleft\varphi\in\mathrm{Prob}(\partial_r M),\quad \gamma(s)\triangleleft\varphi\in\mathrm{Prob}(\partial_s M)
  \end{equation*}
  are the marginals of a probability measure on $M\times M$ supported on
  \begin{equation*}
    \{(p,q)\in M\times M\ |\ d(p,q) = d(\gamma(r),\gamma(s))=s-r\}.
  \end{equation*}
It follows that $\gamma(r)\triangleleft\varphi$ is in fact supported on $\partial_{s\la r}M$. For $s$ sufficiently close to $r$ the uniqueness (\Cref{le.uniq-rs}), for every point in $\partial_{s\la r}M$, of a point in $\partial_sM$ that is $s-r$ away from it then implies that we have 
\begin{equation*}
  \gamma(s)\triangleleft\varphi = (\psi_{s\la r})_*(\gamma(r)\triangleleft\varphi). 
\end{equation*}
We can now repeat the procedure with $s$ in place of $r$ and $s'\in (s,\ell]$. \Cref{le.uniq-rs} ensures that we can choose the differences $s'-s$ to be bounded below by some $\varepsilon>0$ and hence eventually exhaust the interval $[r,\ell]$. All in all, the conclusion will be that  
\begin{equation}\label{eq:11}
  \gamma(\ell)\triangleleft\varphi = (\psi_{\ell\la r})_*(\gamma(r)\triangleleft\varphi),\quad \forall \text{ states $\varphi$ on }\cQ.  
\end{equation}
But this says that the image of $y=\gamma(\ell)$ through $\triangleleft\varphi$ depends only on the image of $x$ through $\varphi$; since the state $\varphi$ on $\cQ$ was arbitrary, this finishes the proof that $y$ is attached to $\gamma(r)\in \partial_rM$.
\end{le.is-att}

As a consequence of \Cref{le.is-att} we have

\begin{corollary}\label{cor.faith}
  Let $\alpha$ be an isometric faithful action of a compact quantum group $\cQ$ on a compact Riemannian manifold $M$, all of whose connected components have non-empty boundary.
  
Then, the actions induced by $\alpha$ on any of the sets $\partial_rM$ for sufficiently small $r>0$ are faithful.
\end{corollary}
\begin{proof}
  This follows from \Cref{pr.att,le.is-att}, which show jointly that every slice $\cQ_y$, $x\in M$ is contained in some other slice $\cQ_x$, $x\in \partial_rM$. Since
  \begin{equation*}
    \cQ_y,\ y\in M
  \end{equation*}
  generate $\cQ$, so do the subalgebras $\cQ_x\subseteq \cQ$, $x\in \partial_rM$, finishing the proof.
\end{proof}

We also record the following consequence of the proof of \Cref{le.is-att}:

\begin{corollary}\label{cor.pres-trnsl}
  If $r\ge 0$ is sufficiently small and $s\ge r$ then the map
  \begin{equation*}
    \psi_{s\la r}:\partial_{s\la r} M\to \partial_s M
  \end{equation*}
  is equivariant for the actions of $\cQ$ on $\partial_{s\la r}M$ and $\partial_sM$ from \Cref{cor.prsv}.  
\end{corollary}
\begin{proof}
  This follows from \Cref{eq:11}.
\end{proof}

Next, we address the smoothness issue for isometric quantum actions on Riemannian manifolds with boundary.

\begin{proposition}\label{pr.bdry-wsmth}
  An isometric action $\alpha$ of a compact quantum group $\cQ$ on a Riemannian manifold $M$ (possibly with boundary) is weakly smooth. %restricts to a weakly smooth action on $\partial_{\ge r}M$ for every sufficiently small $r>0$.
\end{proposition}
\begin{proof}
  The boundary-less case was taken care of in the course of proving \Cref{th.nbdry}, so we focus on the case when $\partial M\ne \emptyset$. We know from \Cref{cor.prsv} that all of the sets described in \Cref{eq:partials} (and the analogues $\partial_{\ge r}M$, etc.) are preserved by $\alpha$. Now fix a small $r>0$. The functions
\begin{equation*}
  D_x\in C^{\infty}(\partial_{\ge r}M),\quad x\in \text{a neighborhood of }\partial_{\ge r}M 
\end{equation*}
are easily seen to satisfy the conclusion of \Cref{le.d_x_density} by a simple adaptation of the proof of that result, so we can conclude as in the proof of \Cref{th.nbdry} that the restriction of $\alpha$ to the invariant submanifold $\partial_{\ge r}M$ is weakly smooth.

Next, for small $r>0$ (small enough so that $2r$ is tame, for instance) consider an automatically-increasing diffeomorphism
\begin{equation*}
  \theta:\bR_{\ge 0}\to \bR_{\ge r}
\end{equation*}
that is the identity on $\bR_{\ge 2r}$. With the help of $\theta$ we can define a ``collar-squeeze'' diffeomorphism
\begin{equation*}
  \theta_{\cat{sq}}:M\cong M_{\ge r},
\end{equation*}
acting as
\begin{equation*}
  \psi_{\theta(s)\la s}:\partial_s M\to \partial_{\theta(s)}M,\ \forall s\in \bR_{\ge 0}
\end{equation*}
for the $\psi_{\bullet\la\bullet}$ maps introduced in \Cref{eq:psis} (so in particular $\theta_{\cat{sq}}$ is the identity on $\partial_{\ge 2r}M$). 

\Cref{cor.prsv} and the $\alpha$-equivariance of the maps $\psi_{\bullet\la\bullet}$ (expressed for instance as \Cref{eq:11}) imply that $\theta_{\cat{sq}}$ is $\alpha$-equivariant. But we have already argued that $\alpha$ is weakly smooth on the image $\partial_{\ge r}$ of $\theta_{\cat{SQ}}$, so since the latter is a diffeomorphism, $\alpha$ must be weakly smooth on the domain $M$ of $\theta_{\cat{SQ}}$ as well. This finishes the proof.   
\end{proof}

Now fix some tame $r>0$, and let $\psi:\bR_{\ge 0}\to \bR$ be a continuous function, constant on $[r,\infty)$. For any $C^*$-algebra $\cQ$ we have a bounded (Banach space) endomorphism
\begin{equation*}
  \cat{sc}_{\psi}:C(M,\cQ)\to C(M,\cQ)
\end{equation*}
that scales a function $f:M\to \cQ$ by $\psi(s)$ along $\partial_s M$. Note that the norm of $\cat{sc}_{\psi}$ is the supremum of $|\psi|$.

Because by \Cref{cor.prsv} $\alpha$ preserves all $\partial_r M$ and the resulting action
\begin{equation*}
  \beta:C(\partial_r M)\to C(\partial_r M, \cQ)
\end{equation*}
is linear, scaling a function $f$ on $M$ along $\partial_rM$ and then applying $\alpha$ results in the scaling of $\alpha(f)$ along $\partial_r M$ by the same amount. In other words, for any $\psi$ as above, $\alpha$ intertwines the two instances of $\cat{sc}_{\psi}$:
\begin{equation}\label{eq:alphaisequiv}
\begin{tikzpicture}[auto,baseline=(current  bounding  box.center)]
\path[anchor=base] 
(0,0) node (l) {$C(M)$}
+(3,.5) node (u) {$C(M)$}
+(3,-.5) node (d) {$C(M,\cQ)$}
+(6,0) node (r) {$C(M,\cQ)$.}
;
\draw[->] (l) to[bend left=6] node[pos=.5,auto] {$\scriptstyle \cat{sc}_{\psi}$} (u);
\draw[->] (u) to[bend left=6] node[pos=.5,auto] {$\scriptstyle \alpha$} (r);
\draw[->] (l) to[bend right=6] node[pos=.5,auto,swap] {$\scriptstyle \alpha$} (d);
\draw[->] (d) to[bend right=6] node[pos=.5,auto,swap] {$\scriptstyle \cat{sc}_{\psi}$} (r);
\end{tikzpicture}
\end{equation}

We are now ready to link the present discussion to the strong-invariance material from \Cref{subse:inv}.

\begin{lemma}\label{le:bdrystrong}
The boundary $\partial M=\partial_0 M$ is in fact {\it strongly} $\alpha$-invariant.   
\end{lemma}
\begin{proof}
  With $X=M$, $Z=\partial M$ and $U=X-Z$ we want to show that the density condition \Cref{eq:zxu} holds. Because $\alpha$ itself is an action, every
  \begin{equation*}
    F\in C_0(U,\cQ)
  \end{equation*}
  (i.e. continuous function $M\to \cQ$ vanishing on the boundary) is approximable arbitrarily well by elements of the form
  \begin{equation*}
    \sum_{i=1}^t \alpha(f_i)(1\otimes x_i),\quad f_i\in C(M),\ x_i\in \cQ.
  \end{equation*}
  Now pick a continuous, non-decreasing $\psi:\bR_{\ge 0}\to \bR_{\ge 0}$ as in the discussion above, which
  \begin{itemize}
  \item equals $1$ for $[r,\infty)$ for some sufficiently small tame $r$ and
  \item vanishes at $0$. 
  \end{itemize}
  Applying the contraction $\cat{sc}_{\psi}$ of $C(M,\cQ)$ to both sides of
  \begin{equation*}
    \sum_{i=1}^t \alpha(f_i)(1\otimes x_i)\quad \simeq_{\varepsilon}\quad F
  \end{equation*}
  will produce
  \begin{itemize}
  \item on the right hand side a function close to $F$ if the $r$ above is sufficiently small, and
  \item on the left hand side
    \begin{equation*}
      \cat{sc}_{\psi}\left(\sum_{i=1}^t \alpha(f_i)(1\otimes x_i)\right) = \sum_{i=1}^t \alpha(\cat{sc}_{\psi}(f_i))(1\otimes x_i),
    \end{equation*}
    using the $\cat{sc}_{\psi}$-equivariance \Cref{eq:alphaisequiv} of $\alpha$.  
  \end{itemize}
  Since $\cat{sc}_{\psi}(f_i)$ belong to $C_0(U)$ (i.e. vanish on $\partial M$) because $\psi$ vanishes at $0$, this finishes the proof.
\end{proof}

It follows from \Cref{le:bdrystrong} that the discussion in \Cref{subse:inv} on gluing actions along common subspaces applies to the case when $X_1=M=X_2$ is a smooth manifold with boundary and
\begin{equation*}
  \iota_i:Z=\partial M\to M
\end{equation*}
are both equal to the inclusion, so that $X=X_1\cup_Z X_2$ is the {\it double} $D(M)$ of $M$ (e.g. \cite[Example 9.32]{lee}). $D(M)$ is a topological boundary-less manifold which can be given a smooth structure compatible with that of (the two copies of) $M$ \cite[Theorem 9.29]{lee}.

The proof of the latter theorem makes it clear that the smooth structure on $D(M)$ depends on a choice of collar neighborhoods of $\partial M$ in the two copies of $M$. For our purposes, we select (on both copies of $M$) a neighborhood adapted to the boundary in the sense of \Cref{def.tame}: one coordinate measures Riemannian distance from the boundary whereas the others are chosen arbitrarily on the boundary and kept constant along geodesics orthogonal to it.

Whenever we refer to $D(M)$ as a smooth manifold we always assume the smooth structure is constructed as described above. Doubling a manifold without boundary simply produces two disjoint copies of it, so that $D(M)$ also contains two copies of each boundary-less component of $M$. Starting with the action $\alpha$ on $M$, we write $\alpha^2$ (``doubled $\alpha$'') for the action on $D(M)$ induced as in \Cref{le.glue}. 

\begin{proposition}\label{th.doubled}
  Let $\alpha$ be an isometric action of $\cQ$ on a Riemannian manifold $M$ with boundary. The doubled action
  \begin{equation*}
    \alpha^2:C(D(M))\to C(D(M),\cQ)
  \end{equation*}
  is weakly smooth and $\alpha^2_r$ is injective. 
\end{proposition}

Consider a collar neighborhood $U=\partial_{<r}M$ of $\partial M$ in $M$ (see \Cref{eq:partials} for the notation) with its adapted coordinate system $(x_1,\cdots,x_n)$ in the sense of \Cref{def.tame}, $x_n$ denoting distance from $\partial M$. Let $\psi:[0,r]\to \bR$ be a continuous (typically smooth) function. We call a function on $U$ {\it $\psi$-separable} if it is of the form
\begin{equation}\label{eq:sepmod}
  (x_1,\cdots,x_n)\mapsto f(x_1,\cdots,x_{n-1})\cdot \psi(x_n). 
\end{equation}
The relevance of the notion to the subsequent discussion is captured by

\begin{lemma}\label{le:sep}
  If $f\in C(M)$ is $\psi$-separable on $U$ for some continuous $\psi$ then $\alpha(f)\in C(M,\cQ)$ is again $\psi$-separable.
\end{lemma}
\begin{proof}
  We work on the $\alpha$-invariant closed collar $X:=\partial_{r\la 0}M$, to simplify the discussion. In the notation introduced prior to \Cref{le:bdrystrong}, the function \Cref{eq:sepmod} is obtained by applying $\cat{SC}_{\psi}$ to a function on $X$ independent of $x_n$ (i.e. a function depending only on the first $n-1$ variables). By the $\alpha$-equivariance \Cref{eq:alphaisequiv} of $\cat{SC}_{\psi}$, it is enough to prove the claim for $x_n$-independent functions, i.e. for $\psi\equiv 1$.

  To that end consider such a function $f\in C(X)$, independent of $x_n$. This means that for $s\in [0,r]$ in the commutative diagram
\begin{equation*}
\begin{tikzpicture}[auto,baseline=(current  bounding  box.center)]
\path[anchor=base] 
(0,0) node (l) {$C(\partial_s M)$}
+(3,.5) node (u) {$C(\partial M)$}
+(3,-.5) node (d) {$C(\partial_s M,\cQ)$}
+(6,0) node (r) {$C(\partial M,\cQ)$}
;
\draw[->] (l) to[bend left=6] node[pos=.5,auto] {$\scriptstyle \psi_{s\la 0}^*$} (u);
\draw[->] (u) to[bend left=6] node[pos=.5,auto] {$\scriptstyle \alpha$} (r);
\draw[->] (l) to[bend right=6] node[pos=.5,auto,swap] {$\scriptstyle \alpha$} (d);
\draw[->] (d) to[bend right=6] node[pos=.5,auto,swap] {$\scriptstyle \psi_{s\la 0}^*\otimes\id_{\cQ}$} (r);
\end{tikzpicture}
\end{equation*}
(where the actions induced by $\alpha$ are denoted by the same symbol) the top left arrow maps
\begin{equation*}
  f|_{\partial_s M}\longmapsto f|_{\partial M}. 
\end{equation*}
But then by the very definition of the induced actions $\alpha$ (the two south-easterly arrows) the bottom right arrow similarly maps 
\begin{equation*}
  \alpha(f)|_{\partial_s M}\longmapsto \alpha(f)|_{\partial M};
\end{equation*}
in turn, this says precisely that, having identified $\partial M$ and $\partial_s M$ via the first $n-1$ coordinates $x_i$, the restrictions of $\alpha(f)$ to the two sets are equal. Since $s\in [0,r]$ was arbitrary, this is what was needed.
\end{proof}

\pf{th.doubled}
\begin{th.doubled}
  The second part (injectivity) follows from the last statement in \Cref{le.glue}, so it remains to prove weak smoothness.

As above, fix a collar neighborhood $U=\partial_{<r}M$ of $\partial M$ with the adapted coordinate system $(x_1,\cdots,x_n)$ that we used in the construction of the smooth structure on $D(M)$ ($x_n$ denoting distance from $\partial M$).

We extend the notion of $\psi$-separability to functions on
\begin{equation*}
  V:=U\cup_{\partial M}U\cong (-r,r)\times \partial M
\end{equation*}
and $\psi:(-r,r)\to \bR$. It follows from the definition of $\alpha^2$ that for every smooth $\psi:(-r,r)\to \bR$ smooth functions on $D(M)$ that are $\psi$-separable on $V$ are sent by $\alpha^2$ to smooth functions $D(M)\to \cQ$ that are $\psi$-separable on $V$.

The weak smoothness of $\alpha^2$ now follows from the Fr\'echet density of the algebra generated by
\begin{equation*}
  \{f\in C^{\infty}(D(M))\ |\ f\text{ is }\psi-\text{separable on }V\}
\end{equation*}
in $C^{\infty}(D(M))$.
\end{th.doubled}

As a consequence, we have the following generalization of \Cref{th.nbdry}.

\begin{theorem}\label{th.rig-bdry}
  Let $M$ be a compact connected Riemannian manifold, possibly with boundary. Then, every faithful compact quantum group action on $M$ isometric with respect to the geodesic distance $d$ is classical.
\end{theorem}
\begin{proof}
  Let $\alpha$ be an action by the compact quantum group $\cQ$ as in the statement and $\alpha^2$ its doubled version. Since $M$ is connected, $D(M)$ is a connected closed manifold. By \Cref{th.doubled} $\alpha^2$ meets the requirements of \Cref{th.prsrv} and hence $\alpha^2$ preserves some Riemannian metric on $D(M)$. But then $\cQ$ must be classical by \Cref{th.nbdry}, finishing the proof.
\end{proof}

%%%%%%%%%%%%%%%%%%%%%%%%%%%%%%%%%%%%%%%%%%%%%%%%%%%%%%%%%%%%%%%%%%%%%%%%%%%%%
%%%%%%%%%%%%%%%%%%%%%%%%%%%%%%%%%%%%%%%%%%%%%%%%%%%%%%%%%%%%%%%%%%%%%%%%%%%%%
\section{Uniformly distributed measures}\label{subse.unif}

Another situation when quantum isometry groups exist automatically (though they may not be classical, in general) occurs when the metric space is equipped with a probability measure as in the title of the present subsection. We first recall that concept (see e.g. \cite[Definition 3.3]{mat}). 

\begin{definition}\label{def.ud}
  A measure on a metric space $X$ is {\it uniformly distributed} (or {\it UD} for short) if
  \begin{equation*}
    \mu(B(x,r)) = \mu(B(y,r)),\ \forall x,y\in X,\ \forall r\in \bR_{\ge 0}. 
  \end{equation*}
  In other words, the measure assigns equal mass to balls of equal radius, regardless of center. 
\end{definition}

Uniformly distributed measures on compact metric spaces are unique up to scaling when they exist \cite[Theorem 3.4]{mat}, and hence UD probability measures are unique (or non-existent).

Now let $\mu$ be a UD probability measure on $(X,d)$ and consider a CQG action
\begin{equation*}
  C(X)\to C(X)\otimes C(G)
\end{equation*}
on $X$ that is isometric in the sense of \cite[Definition 3.1]{metric}. The following auxiliary observation will be used later. 

%We use the notation and conventions of \cite{Chi15}.

\begin{lemma}\label{le.munu}
  Let $\mu$ be a UD probability measure on the compact metric space $(X,d)$ and $\nu$ any probability measure. Then, for every $r\in \bR_{\ge 0}$ we have  
  \begin{equation*}
    \int_X \nu(B(x,r))\ \mathrm{d}\mu(x)=\mu_r:=\mu(B(x,r)),\ \forall x\in X.
  \end{equation*}
\end{lemma}
\begin{proof}
  By Fubini's theorem, the left hand side is
  \begin{align*}
    \int_{X\times X}\chi_{B(x,r)}(y)\ \mathrm{d}\nu(y)\ \mathrm{d}\mu(x) &=  \int_{X\times X}\chi_{B(y,r)}(x)\ \mathrm{d}\nu(y)\ \mathrm{d}\mu(x)\\
    &=\int_X \mu_r\ \mathrm{d}\nu(y) = \mu_r. 
  \end{align*}
  This finishes the proof.
\end{proof}

\begin{theorem}\label{th.ud}
  A uniformly distributed measure $\mu$ on a compact metric space $(X,d)$ is automatically invariant under any isometric CQG action $\alpha$.
\end{theorem}
\begin{proof}
  We have to show that for every state $\varphi$ on $C(G)$ and UD probability measure $\mu$ $(X,d)$ we have
  \begin{equation*}
    \mu\triangleleft\varphi = \mu. 
  \end{equation*}
  Lift $\alpha$ to a coaction (denoted slightly abusively by the same symbol)
  \begin{equation*}
    \alpha:W(X)\to W(X)\otimes C(G)''
  \end{equation*}
  where 
  \begin{itemize}
  \item $C(G)''$ is the von Neumann algebra generated by $C(G)$ in its Haar-state GNS representation;
  \item $W(X)$ is the von Neumann hull of $C(X)$. 
  \end{itemize}
  As in \cite[\S 3]{Chi15}, for a point $x\in X$ and a Borel subset $S\subseteq X$ we denote by $a_{x;S}$ the image of the characteristic function $\chi_S$ through $(\mathrm{ev}_x\otimes \id)\alpha$. According to \cite[equation (13)]{Chi15} we have
  \begin{equation}\label{eq:1}
    a_{x;B(y,r)} = \kappa(a_{y;B(x,r)})
  \end{equation}
  for all pairs of points $x,y\in X$ and radii $r\in \bR_{\ge 0}$. By the very definition of the action $\triangleleft$ of the state semigroup $\mathrm{Prob}(G)$ on $\mathrm{Prob}(X)$, we have
  \begin{equation*}
    (\mu\triangleleft\varphi)(B(y,r)) = \int_X \varphi(a_{x;B(y,r)})\ \mathrm{d}\mu(x),
  \end{equation*}
  i.e. the integral of the left hand side of \Cref{eq:1} against $\mu(x)$. Using \Cref{eq:1}, this is also
  \begin{equation*}
    \int_X(\mathrm{ev}_y\triangleleft\varphi\circ\kappa)(B(x,r))\ \mathrm{d}\mu(x). 
  \end{equation*}
  Applying \Cref{le.munu} with $\nu=\mathrm{ev}_y\triangleleft \varphi\circ\kappa$ we conclude that this equals $\mu_r$. In conclusion,
  \begin{equation*}
    (\mu\triangleleft\varphi)(B(y,r)) = \mu_r = \mu(B(y,r)), \forall y\in X. 
  \end{equation*}
  This finishes the proof.
\end{proof}

The reason why this has a bearing on the existence of $QISO(X,d)$ is encapsulated by the following result.

\begin{theorem}\label{th.xdmu}
  Let $(X,d)$ be a compact metric space and $\mu$ a Borel probability measure with full support. Then, there is a universal compact quantum group $QISO(X,d,\mu)$ acting on $(X,d)$ isometrically and preserving $\mu$. 
\end{theorem}

We need some preparation. Throughout the discussion, we assume $\alpha$ is an isometric, $\mu$-preserving action on $(X,d,\mu)$. First, consider the integral operator with kernel $d$, i.e.
\begin{equation}\label{eq:intop}
  K: f\mapsto \int_Xd(-,x)f(x)\ \mathrm{d}\mu(x). 
\end{equation}
It can be regarded as an operator on either the Banach space $C(X)$ or the Hilbert space $L^2(X,\mu)$, and it is compact in either guise ( self-adjoint in the latter). Working in the $L^2$-picture, its non-zero eigenspaces
\begin{equation*}
  V_{\lambda}:=\ker(K-\lambda),\ \lambda\ne 0
\end{equation*}
coincide on $C(X)$ and $L^2(X,\mu)$ because its image consists of continuous functions. 

\begin{lemma}\label{le:invmeas}
  Let $(X,d,\mu)$ be a compact metric space equipped with a Borel probability measure and $\alpha:\cC\to \cC\otimes \cQ$ an isometric action as in \Cref{def.isometric} leaving $\mu$ invariant. Then, for any two continuous functions $f,g\in \cC$ we have
  \begin{equation}\label{eq:bal}
    \int_X \kappa(\alpha g(x)) f(x)\ \mathrm{d}\mu(x) = \int_X g(x) \alpha f(x)\ \mathrm{d}\mu(x) \in \cQ.
  \end{equation}
\end{lemma}
\begin{proof}
  It is enough to work with the dense $*$-subalgebras $\cC_0\subseteq \cC$ and $\cQ_0\subseteq \cQ$ discussed in \Cref{subse:cqg}, so that we can use Sweedler notation for the comultiplication and coaction:
  \begin{align*}
    \cC_0\ni f &\stackrel{\alpha}{\longmapsto} f_0\otimes f_1 \in \cC_0\otimes_{\rm alg} \cQ_0\\
    \cQ_0\ni x& \stackrel{\Delta}{\longmapsto} x_1\otimes x_2\in \cQ_0\otimes_{\rm alg} \cQ_0,
  \end{align*}
  etc. The desired equality \Cref{eq:bal} then reads
  \begin{equation*}
    \mu(fg_0)\kappa(g_1) = \mu(f_0g) f_1. 
  \end{equation*}
  To see why this is so, use the $\alpha$-invariance of $\mu$,
  \begin{equation*}
    \mu(x_0)x_1 = \mu(x)1,
  \end{equation*}
  on $x=fg_0$ to obtain
  \begin{equation*}
    \mu(fg_0)\kappa(g_1) = \mu(f_0 g_0)f_1g_1\kappa(g_2) = \mu(f_0g) f_1,
  \end{equation*}
  where the last equality uses the defining property of the antipode in a Hopf algebra, namely
  \begin{equation*}
    g_1\kappa(g_2) = \varepsilon(g)1. 
  \end{equation*}
  This finishes the proof.
\end{proof}

\begin{lemma}\label{le:comm}
  The integral operator $\cK$ intertwines the action $\alpha$ in the sense that
  \begin{equation}\label{eq:ccqq}
 \begin{tikzpicture}[auto,baseline=(current  bounding  box.center)]
  \path[anchor=base] 
  (0,0) node (l) {$\cC$}
  +(3,.5) node (u) {$\cC\otimes \cQ$}
  +(3,-.5) node (d) {$\cC$}
  +(6,0) node (r) {$\cC\otimes \cQ$}
  ;
  \draw[->] (l) to[bend left=6] node[pos=.5,auto] {$\scriptstyle \alpha$} (u);
  \draw[->] (l) to[bend right=6] node[pos=.5,auto,swap] {$\scriptstyle K$} (d);
  \draw[->] (d) to[bend right=6] node[pos=.5,auto,swap] {$\scriptstyle \alpha$} (r);
  \draw[->] (u) to[bend left=6] node[pos=.5,auto] {$\scriptstyle K\otimes \id$} (r);
\end{tikzpicture}
\end{equation}
commutes. 
\end{lemma}
\begin{proof}
  For an arbitrary $f\in \cC$ we will evaluate the images of $f$ through both the upper and lower paths in \Cref{eq:ccqq} at a fixed point $y\in X$. On the one hand, we have
  \begin{align*}
    \alpha(Kf)(y) &= \int_X \alpha(d_x)(y) f(x)\ \mathrm{d}\mu(x)\\
                  &=\int_X \kappa(\alpha(d_y)(x)) f(x)\ \mathrm{d}\mu(x),\\
  \end{align*}
  using the fact that $\alpha$ is isometric.

  On the other hand,
  \begin{align*}
    (K\otimes \id)(\alpha f)(y) &= \int_X d_x(y) \alpha f(x)\ \mathrm{d}\mu(x)\\
                                &=\int_X d_y(x) \alpha f(x)\ \mathrm{d}\mu(x).\\
  \end{align*}
  That these two are equal now follows by applying \Cref{le:invmeas} with $g=d_y$.
\end{proof}

\pf{th.xdmu}
\begin{th.xdmu}
  Let $\cQ$ be a compact quantum group acting isometrically via
  \begin{equation*}
    \alpha:C(X)\to C(X)\otimes \cQ
  \end{equation*}
  on $(X,d)$ and preserving $\mu$, and consider the integral operator \Cref{eq:intop} on $L^2(X,\mu)$. Because by \Cref{le:comm} $K$ is an intertwiner for the action $\alpha$, the latter preserves the non-zero finite-dimensional eigenspaces $V_{\lambda}$, $\lambda\ne 0$ of $K$. Moreover, since $K$ is self-adjoint, the closed span of the $V_{\lambda}$ coincides with the closure of the range of $K$.

  Applying $K$ to bump functions $\psi$ localized near points $y\in X$ we can approximate
  \begin{equation*}
    d_y:=d(y,-)\simeq K\psi
  \end{equation*}
  arbitrarily well, so the $*$-algebra $\cA\subset C(X)$ generated by $V_{\lambda}$, $\lambda\ne 0$ is dense.

  Now consider the lattice $\cL$ of subspaces of $C(X)$ generated by the $V_{\lambda}$, $\lambda\ne 0$ $\bC 1$ and closed under the following operations
  \begin{itemize}
  \item taking products: if $V_i\in \cL$ for $1\le i\le t$ then
    \begin{equation*}
      V_1\cdot\ldots\cdot V_t\in \cL.
    \end{equation*}
  \item taking adjoints:
    \begin{equation*}
      V\in \cL\Rightarrow V^*\in \cL.
    \end{equation*}
  \item taking orthogonal complements with respect to the inner product induced by $\mu$: if $V\subseteq W$ both belong to $\cL$ then so does
    \begin{equation*}
      W\ominus V:=V^{\perp}\cap W. 
    \end{equation*}
  \end{itemize}
  The minimal (non-zero) elements of $\cL$ are then finite-dimensional subspaces preserved by the action, whose direct sum is precisely the $*$-subalgebra $\cA\subset C(X)$. Furthermore, these spaces constitute an {\it orthogonal filtration} $V_i$, $i\in \cI$ for $C(X)$ with respect to the state $\mu$ on it in the sense of \cite[Definition 2.1]{ban-sk}.

  It follows from \cite[Theorem 2.7]{ban-sk} that the is a universal compact quantum group
  \begin{equation*}
    QISO(C(X),\mu,(V_i)_{i\in \cI})
  \end{equation*}
  acting on $X$ in a filtration-preserving manner, and from \cite[Theorem 4.4]{Chi15} that the latter has a largest compact quantum subgroup $\cQ_u$ acting isometrically. The argument above shows that the action of $\cQ$ on $X$ factors through that of $\cQ_u$, i.e. that the latter has the defining universality property of $QISO(X,d,\mu)$.
\end{th.xdmu}

In particular, we have

\begin{corollary}\label{cor.ud-implies-qiso}
A compact metric space $(X,d)$ admitting a uniformly distributed probability measure admits a quantum isometry group $QISO(X,d)$.   
\end{corollary}
\begin{proof}
Immediate from \Cref{th.ud,th.xdmu}.
\end{proof}

For instance:

\begin{corollary}\label{cor.homog-sp}
  Let $G$ be a compact group and $X$ a homogeneous $G$-space equipped with a $G$-invariant metric $d$. Then, there is a universal quantum group $QISO(X,d)$ of isometries of $(X,d)$. 
\end{corollary}
\begin{proof}
  This is a consequence of \Cref{th.ud}, since $(X,d)$ admits a UD probability measure: simply select {\it any} probability measure on $X$ and average it with respect to the Haar measure of $G$.
\end{proof}

%%%%%%%%%%%%%%%%%%%%%%%%%%%%%%%%%%%%%%%%%%%%%%%%%%%%%%%%%%%%%%%%%%%%%%%%%%%%%%%%%%%%%%%%%%%%%%%%%%%%%%%%%%%%%%%%%%
%%%%%%%%%%%%%%%%%%%%%%%%%%%%%%%%%%%%%%%%%%%%%%%%%%%%%%%%%%%%%%%%%%%%%%%%%%%%%%%%%%%%%%%%%%%%%%%%%%%%%%%%%%%%%%%%%%

%\bibliography{bib}{}

\begin{thebibliography}{10}

\bibitem{abb-cauchy}
Stephanie~B. Alexander, I.~David Berg, and Richard~L. Bishop.
\newblock Cauchy uniqueness in the {R}iemannian obstacle problem.
\newblock In {\em Differential geometry, {P}e\~{n}\'{\i}scola 1985}, volume
  1209 of {\em Lecture Notes in Math.}, pages 1--7. Springer, Berlin, 1986.

\bibitem{ban_1}
Teodor Banica.
\newblock Quantum automorphism groups of small metric spaces.
\newblock {\em Pacific J. Math.}, 219(1):27--51, 2005.

\bibitem{ban_col}
Teodor Banica and Beno\^{i}t Collins.
\newblock Integration over quantum permutation groups.
\newblock {\em J. Funct. Anal.}, 242(2):641--657, 2007.

\bibitem{ban-sk}
Teodor Banica and Adam Skalski.
\newblock Quantum symmetry groups of {$C^\ast$}-algebras equipped with
  orthogonal filtrations.
\newblock {\em Proc. Lond. Math. Soc. (3)}, 106(5):980--1004, 2013.

\bibitem{skalski_bhow}
Jyotishman Bhowmick and Adam Skalski.
\newblock Quantum isometry groups of noncommutative manifolds associated to
  group {$C^*$}-algebras.
\newblock {\em J. Geom. Phys.}, 60(10):1474--1489, 2010.

\bibitem{bichon}
Julien Bichon.
\newblock Quantum automorphism groups of finite graphs.
\newblock {\em Proc. Amer. Math. Soc.}, 131(3):665--673 (electronic), 2003.

\bibitem{ccl}
S.~S. Chern, W.~H. Chen, and K.~S. Lam.
\newblock {\em Lectures on differential geometry}, volume~1 of {\em Series on
  University Mathematics}.
\newblock World Scientific Publishing Co., Inc., River Edge, NJ, 1999.

\bibitem{Chi15}
Alexandru Chirvasitu.
\newblock On quantum symmetries of compact metric spaces.
\newblock {\em J. Geom. Phys.}, 94:141--157, 2015.

\bibitem{chirvasitu}
Alexandru Chirvasitu.
\newblock Quantum rigidity of negatively curved manifolds.
\newblock {\em Comm. Math. Phys.}, 344(1):193--221, 2016.

\bibitem{con}
Alain Connes.
\newblock {\em Noncommutative geometry}.
\newblock Academic Press, Inc., San Diego, CA, 1994.

\bibitem{doC92}
Manfredo~Perdig{\~a}o do~Carmo.
\newblock {\em Riemannian geometry}.
\newblock Mathematics: Theory \& Applications. Birkh\"auser Boston, Inc.,
  Boston, MA, 1992.
\newblock Translated from the second Portuguese edition by Francis Flaherty.

\bibitem{drinfeld}
V.~G. Drinfel'd.
\newblock Quantum groups.
\newblock In {\em Proceedings of the {I}nternational {C}ongress of
  {M}athematicians, {V}ol. 1, 2 ({B}erkeley, {C}alif., 1986)}, pages 798--820.
  Amer. Math. Soc., Providence, RI, 1987.

\bibitem{Etingof_walton_1}
Pavel Etingof and Chelsea Walton.
\newblock Semisimple {H}opf actions on commutative domains.
\newblock {\em Adv. Math.}, 251:47--61, 2014.

\bibitem{frt}
L.~Faddeev, N.~Reshetikhin, and L.~Takhtajan.
\newblock Quantum groups.
\newblock In {\em Braid group, knot theory and statistical mechanics}, volume~9
  of {\em Adv. Ser. Math. Phys.}, pages 97--110. World Sci. Publ., Teaneck, NJ,
  1989.

\bibitem{injective}
D~{Goswami} and S~{Joardar}.
\newblock {A note on the injectivity of action by compact quantum groups on a
  class of $C^{\ast}$-algebras}.
\newblock {\em arXiv e-prints}, June 2018.

\bibitem{Goswami}
Debashish Goswami.
\newblock Quantum group of isometries in classical and noncommutative geometry.
\newblock {\em Comm. Math. Phys.}, 285(1):141--160, 2009.

\bibitem{metric}
Debashish Goswami.
\newblock Existence and examples of quantum isometry groups for a class of
  compact metric spaces.
\newblock {\em Adv. Math.}, 280:340--359, 2015.

\bibitem{final}
Debashish Goswami.
\newblock Non-existence of genuine (compact) quantum symmetries of compact,
  connected smooth manifolds.
\newblock {\em Adv. Math.}, 369:107181, 2020.

\bibitem{gafa}
Debashish Goswami and Soumalya Joardar.
\newblock Non-existence of faithful isometric action of compact quantum groups
  on compact, connected {R}iemannian manifolds.
\newblock {\em Geom. Funct. Anal.}, 28(1):146--178, 2018.

\bibitem{grom-struct}
Misha Gromov.
\newblock {\em Metric structures for {R}iemannian and non-{R}iemannian spaces}.
\newblock Modern Birkh\"{a}user Classics. Birkh\"{a}user Boston, Inc., Boston,
  MA, english edition, 2007.
\newblock Based on the 1981 French original, With appendices by M. Katz, P.
  Pansu and S. Semmes, Translated from the French by Sean Michael Bates.

\bibitem{hua-inv}
Huichi Huang.
\newblock Invariant subsets under compact quantum group actions.
\newblock {\em J. Noncommut. Geom.}, 10(2):447--469, 2016.

\bibitem{jimbo}
Michio Jimbo.
\newblock Solvable lattice models and quantum groups.
\newblock In {\em Proceedings of the {I}nternational {C}ongress of
  {M}athematicians, {V}ol. {I}, {II} ({K}yoto, 1990)}, pages 1343--1352. Math.
  Soc. Japan, Tokyo, 1991.

\bibitem{kustermans}
Johan Kustermans and Stefaan Vaes.
\newblock Locally compact quantum groups.
\newblock {\em Ann. Sci. \'{E}cole Norm. Sup. (4)}, 33(6):837--934, 2000.

\bibitem{lang-diff}
Serge Lang.
\newblock {\em Differential and {R}iemannian manifolds}, volume 160 of {\em
  Graduate Texts in Mathematics}.
\newblock Springer-Verlag, New York, third edition, 1995.

\bibitem{lee}
John~M. Lee.
\newblock {\em Introduction to smooth manifolds}, volume 218 of {\em Graduate
  Texts in Mathematics}.
\newblock Springer, New York, second edition, 2013.

\bibitem{Soltan}
Jan Liszka-Dalecki and Piotr~M. So\l~tan.
\newblock Quantum isometry groups of symmetric groups.
\newblock {\em Internat. J. Math.}, 23(7):1250074, 25, 2012.

\bibitem{llav}
Jos\'{e}~G. Llavona.
\newblock {\em Approximation of continuously differentiable functions}, volume
  130 of {\em North-Holland Mathematics Studies}.
\newblock North-Holland Publishing Co., Amsterdam, 1986.
\newblock Notas de Matem\'{a}tica [Mathematical Notes], 112.

\bibitem{Van}
Ann Maes and Alfons Van~Daele.
\newblock Notes on compact quantum groups.
\newblock {\em Nieuw Arch. Wisk. (4)}, 16(1-2):73--112, 1998.

\bibitem{manin_book}
Yu.~I. Manin.
\newblock {\em Quantum groups and noncommutative geometry}.
\newblock Universit\'{e} de Montr\'{e}al, Centre de Recherches
  Math\'{e}matiques, Montreal, QC, 1988.

\bibitem{mat}
Pertti Mattila.
\newblock {\em Geometry of sets and measures in {E}uclidean spaces}, volume~44
  of {\em Cambridge Studies in Advanced Mathematics}.
\newblock Cambridge University Press, Cambridge, 1995.
\newblock Fractals and rectifiability.

\bibitem{pv}
S.~{Pigola} and G.~{Veronelli}.
\newblock {The smooth Riemannian extension problem}.
\newblock {\em arXiv e-prints}, June 2016.

\bibitem{Podles}
Piotr Podle\'{s}.
\newblock Symmetries of quantum spaces. {S}ubgroups and quotient spaces of
  quantum {${\rm SU}(2)$} and {${\rm SO}(3)$} groups.
\newblock {\em Comm. Math. Phys.}, 170(1):1--20, 1995.

\bibitem{rud-fa}
Walter Rudin.
\newblock {\em Functional analysis}.
\newblock International Series in Pure and Applied Mathematics. McGraw-Hill,
  Inc., New York, second edition, 1991.

\bibitem{soi}
Ya.~S. So\u{\i}bel'man and L.~L. Vaksman.
\newblock On some problems in the theory of quantum groups.
\newblock In {\em Representation theory and dynamical systems}, volume~9 of
  {\em Adv. Soviet Math.}, pages 3--55. Amer. Math. Soc., Providence, RI, 1992.

\bibitem{trev}
Fran\c{c}ois Tr\`eves.
\newblock {\em Topological vector spaces, distributions and kernels}.
\newblock Dover Publications, Inc., Mineola, NY, 2006.
\newblock Unabridged republication of the 1967 original.

\bibitem{wang}
Shuzhou Wang.
\newblock Quantum symmetry groups of finite spaces.
\newblock {\em Comm. Math. Phys.}, 195(1):195--211, 1998.

\bibitem{Pseudogroup}
S.~L. Woronowicz.
\newblock Compact matrix pseudogroups.
\newblock {\em Comm. Math. Phys.}, 111(4):613--665, 1987.

\bibitem{Wor98}
S.~L. Woronowicz.
\newblock Compact quantum groups.
\newblock In {\em Sym\'etries quantiques ({L}es {H}ouches, 1995)}, pages
  845--884. North-Holland, Amsterdam, 1998.

\end{thebibliography}
\bibliographystyle{plain}
\addcontentsline{toc}{section}{References}

\def\polhk#1{\setbox0=\hbox{#1}{\ooalign{\hidewidth
  \lower1.5ex\hbox{`}\hidewidth\crcr\unhbox0}}}

\Addresses

\end{document}